\documentclass[11pt]{amsart}
\usepackage{amssymb,amscd,verbatim,graphicx}
\usepackage[T1]{fontenc}
\usepackage{mathrsfs}
\usepackage{xy}
\usepackage{bbm} 
\usepackage{url}
\usepackage[colorlinks,linkcolor=blue,citecolor=blue,urlcolor=red]{hyperref}
\xyoption{all}
\swapnumbers
\numberwithin{equation}{section}
\newtheorem{thm}{Theorem}[subsection]
\newtheorem{propose}[thm]{Proposition}
\newtheorem{lemma}[thm]{Lemma}
\newtheorem{cor}[thm]{Corollary}
\theoremstyle{definition}
\newtheorem{defn}[thm]{Definition}

\newtheorem{remark}[thm]{Remark}
\newtheorem{remarks}[thm]{Remarks}

\newtheorem{examples}[thm]{Examples}

\newcommand{\uOm}{\underline{\Omega}}
\newcommand{\T}{\mathrm{T}}
\newcommand{\A}{\mathcal{M}_0}        
\newcommand{\M}{\mathcal{M}_1}        
\newcommand{\tM}{{}^t\!\mathcal{M}_{1}}

\newcommand{\dr}{\mathrm{dR}} 
\newcommand{\ab}{\mathrm{ab}}

\renewcommand{\d}{{\text{\LARGE $\cdot $}}}

\newcommand{\Spec}{\operatorname{Spec}} 
\newcommand{\Hom}{\operatorname{Hom}}      
\newcommand{\Ext}{\operatorname{Ext}}      
\newcommand{\Biext} {\operatorname{Biext}}  
\newcommand{\DM}{{\operatorname{DM}}}  
\newcommand{\DA}{{\operatorname{DA}}}

\newcommand{\MHS}{{\operatorname{MHS}}}

\newcommand{\Shv}{\operatorname{Shv}}
\newcommand{\Sm}{\operatorname{Sm}}

\newcommand{\An}{\operatorname{An}}

\newcommand{\car}{\operatorname{char}}
\newcommand{\Div}{\operatorname{Div}}

\newcommand{\eff}{{\operatorname{eff}}}
\newcommand{\ihom}{{\rm\underline{Hom}}}  

\newcommand{\Aff}{\mathbb{A}}   
\newcommand{\V}{\mathbb{V}}
 
\newcommand{\Lie}{\operatorname{Lie}}
 
\newcommand{\drb}{\mathrm{dRB}}
\newcommand{\bdr}{\mathrm{BdR}}

\newcommand{\Ab}{\mathrm{Mod}_{\Z}}
\newcommand{\Mod}{\mathrm{Mod}_{\Z,K}}
\newcommand{\Modc}{\mathrm{Mod}_{\Z,K}^{\cong}}
\newcommand{\QMod}{\mathrm{Mod}_{\Q,K}}
\newcommand{\cQMod}{\mathrm{Mod}_{K,\Q}}
\newcommand{\fMod}{\mathrm{FMod}_{\Z,K}}
\newcommand{\cMod}{\mathrm{Mod}_{K, \Z}}

\newcommand{\C}{\mathbb{C}}     
\newcommand{\F}{\mathbb{F}}
\newcommand{\Q}{\mathbb{Q}}     
\newcommand{\Z}{\mathbb{Z}}     

\newcommand{\N}{\mathbb{N}}
\newcommand{\G}{\mathbb{G}}     
\newcommand{\EExt}{{\rm \mathbb{E}xt}} 
\newcommand{\R}{\mathbb{R}}     

\newcommand{\im}{\operatorname{Im}}        
\renewcommand{\ker}{\operatorname{Ker}}  
\newcommand{\gr}{{\operatorname{gr}}}        
\newcommand{\Pic}{\operatorname{Pic}}     
\newcommand{\RPic}{\operatorname{RPic}}     
\newcommand{\LAlb}{\operatorname{LAlb}}     

\newcommand{\LA}[1]{\mbox{${\rm L}_{#1}{\rm Alb}$}}
\newcommand{\RA}[1]{\mbox{${\rm R}^{#1}{\rm Pic}$}}

\newcommand{\Tot}{\operatorname{Tot}}     
\newcommand{\NS}  {\operatorname{NS}}      

\newcommand{\tr}{{\operatorname{tr}}}        
\newcommand{\qi}{{\rm q.i.}\,}      

\newcommand{\by}[1]{\stackrel{#1}{\rightarrow}}
\newcommand{\longby}[1]{\stackrel{#1}{\longrightarrow}}
\newcommand{\vlongby}[1]{\stackrel{#1}{\mbox{\large{$\longrightarrow$}}}}
\newcommand{\iso}{\stackrel{\sim}{\longrightarrow}}

\newcommand{\eprooff}{\hfill$\Box$\par}
\renewcommand{\tilde}{\widetilde}
\newcommand{\df}{\mbox{\,${:=}$}\,}
\newcommand{\ie}{{\it i.e.}, }
\newcommand{\cf}{{\it cf. }}
\newcommand{\eg}{{\it e.g. }}

\newcommand{\et} {{\rm \acute{e}t}}
\newcommand{\eh} {{\rm \acute{e}h}}
\newcommand{\Zar}{{\rm Zar}}
\newcommand{\Nis}{{\rm Nis}}
\newcommand{\cdh}{{\rm cdh}}
\newcommand{\an}{{\rm an}}

\newcommand{\fr}{{\rm fr}}
\newcommand{\tor}{{\rm tor}}

\newcommand{\tf}{{\rm tf}}

\newcommand{\gm}{{\rm gm}}

\renewcommand{\bar}{\overline}
\newcommand{\into}{\hookrightarrow}
\renewcommand{\implies}{\mbox{$\Rightarrow$}}

\renewcommand{\lim}{\varprojlim}

\newcommand{\onto}{\mbox{$\to\!\!\!\!\to$}}

\newcommand{\boxtensor}{\def\boxtimesten{\Box\kern-7.59pt\raise1.2pt
\hbox{$\times$} }}                                  

\newcounter{elno}                   

\newcommand{\cA}{\mathcal{A}}

\newcommand{\cC}{\mathcal{C}}

\newcommand{\cO}{\mathcal{O}}

\renewcommand{\phi}{\varphi}
\renewcommand{\epsilon}{\varepsilon}
\renewcommand*{\d}{{\scalebox{0.5}{$\bullet$}}}
\newcommand{\sK}{\mathcal{K}}

\newcommand{\sX}{\mathcal{X}}

\newcommand{\Ker}{\operatorname{Ker}}
\newcommand{\Coker}{\operatorname{Coker}}

\newenvironment{thlist}{\begin{list}{\rm{(\roman{enumi})}}%
{\usecounter{enumi}}}%
{\end{list}}

\begin{document}
\title{Motivic periods and Grothendieck arithmetic invariants}
\author{F. Andreatta, L. Barbieri-Viale}
\address{Dipartimento di Matematica ``F. Enriques'', Universit{\`a} degli Studi di Milano\\ via C. Saldini, 50\\ Milano I-20133\\ Italy}\email{Fabrizio.Andreatta@unimi.it} \email{Luca.Barbieri-Viale@unimi.it}
\author{A. Bertapelle (Appendix by B. Kahn)}
\address{Dipartimento di Matematica ``T. Levi-Civita'', Universit\`a degli Studi di Padova\\ via Trieste, 63\\ Padova I-35121\\ Italy} \email{Alessandra.Bertapelle@unipd.it}
\address{IMJ-PRG\\Case 247\\4 place Jussieu\\75252
Paris Cedex 05\\France}
\email{Bruno.Kahn@imj-prg.fr}

\begin{abstract}
We construct a period regulator for motivic cohomology of an algebraic scheme over a subfield of the complex numbers. For the field of algebraic numbers we formulate a period conjecture for motivic cohomology by saying that this period regulator is surjective. Showing that a suitable Betti--de Rham realization of 1-motives is fully faithful we can verify this period conjecture in several cases. The divisibility properties of motivic cohomology imply that our conjecture is a neat generalization of the classical Grothendieck period conjecture for algebraic cycles on smooth and proper schemes. These divisibility properties are treated in an appendix by B. Kahn (extending previous work of Bloch and Colliot-Th\'el\`ene--Raskind). 
\end{abstract}

\maketitle

\tableofcontents

\section*{Introduction}
Let $X$ be a scheme which is separated and of finite type over a subfield $K$ of the complex numbers.  Consider the $q$-twisted singular cohomology $H^p(X_\an, \Z_\an(q))$ of the analytic space $X_\an$ associated to the base change of $X$ to $\C$ and the $p$th de Rham cohomology $H^p_\dr (X)$, which is an algebraically defined $K$-vector space. We have the following natural $\C$-linear isomorphism 
$$\varpi^{p, q}_X: H^p(X_\an, \Z_\an (q))\otimes_\Z \C\cong H^p_\dr (X)\otimes_K \C$$ 
providing a comparison between these cohomology theories. As Grothendieck originally remarked, for $X$ defined over the field of algebraic numbers $K =\bar \Q$ or a number field, the position of the whole $H^p_\dr (X)$ with respect to $H^p(X_\an, \Z_\an (q))$ under $\varpi^{p, q}_X$ <<yields an interesting arithmetic invariant, generalizing the ``periods'' of regular differential forms>> (see \cite[p. 101 \& footnotes (9) and (10)]{GrdR}, \cf \cite[\S 7.5 \& Chap. 23]{An},  \cite{Bo}, \cite{BC} and \cite[Chap. 5 \& 13]{Hu}).  For the comparison of several notions of\, ``periods'' and versions of the period conjecture we refer to Huber survey \cite{HP} (see also \cite[\S 2.2.2]{BC}). 

The main goal of this paper is to describe this arithmetic invariant,  at least for $p=1$ and all twists, notably, $q=1$ and $q=0$. In more details, we first reconstruct $\varpi^{p, q}_X$ (in Definition~\ref{classper}) by making use of Ayoub's period isomorphism (see Lemma~\ref{APLemma}) in Voevodsky's triangulated category $\DM_\et^\eff$ of motivic complexes for the \'etale topology. Denote by $H^{p, q}_\varpi (X)$ the named arithmetic invariant, \ie the subgroup of those cohomology classes in $H^p(X_\an, \Z_\an (q))$ which are landing in $H^p_\dr (X)$ via $\varpi^{p, q}_X$. We then show the existence of a regulator map (see Corollary \ref{omegacom} and Definition~\ref{defper})
$$ r_\varpi^{p,q}:H^{p,q}(X) \to H^{p,q}_\varpi (X)$$
from \'etale motivic cohomology groups $H^{p,q}(X)$. Here we  regard motivic cohomology canonically identified with $H^p_\eh(X, \Z(q))$ where $ \Z(q)$ is the Suslin-Voevodsky motivic complex (see \cite[Def. 3.1]{VL}), as a complex of sheaves for the $\eh$-topology (introduced in \cite[\S 10.2]{BVK}). We are mostly interested in the case of $q=0, 1$ so that $\Z(0) \cong\Z[0]$ and $\Z(1)\cong \G_m [-1]$ by a theorem of Voevodsky (see \cite[Thm. 4.1]{VL}). 

Following Grothendieck's idea, we conjecture that the period regulator $r_\varpi^{p,q}$ is surjective over $\bar \Q$ and we actually show some evidence.  We easily see that $H^{0,q}_\varpi (X)=0$ for $q\neq 0$ and $r_\varpi^{0,0}$ is an isomorphism: therefore, the first non-trivial case is for $p=1$. Moreover, by making use of Suslin-Voevodsky rigidity theorem we can show that $r_\varpi^{p,q}$ is surjective on torsion (see Lemma~\ref{torhyp}). We can also show: if the vanishing $H^{p,q}(X)\otimes\Q/\Z = 0$ holds true then the surjectivity of $r_\varpi^{p,q}$ is equivalent to the vanishing $H^p_\dr(X)\cap H^p(X_\an, \Q_\an (q))=0$. The divisibility properties of motivic cohomology (see Appendix~\ref{KD}) imply that our conjecture is a neat generalization of the classical period conjecture for algebraic cycles on smooth and proper schemes (see Proposition~\ref{perconjsm}). 
 
In order to study the case $p=1$ we can make use of the description of $H^1$ via the Albanese $1$-motive $\LA{1} (X)$. Recall the existence of the homological motivic Albanese complex $\LAlb(X)$, a complex of 1-motives whose $p$th homology $\LA{p} (X)$ is a 1-motive with cotorsion (see \cite[\S 8.2]{BVK} for details). We can regard complexes of 1-motives as objects of $\DM_\et^\eff$ and by the adjunction properties of $\LAlb$ (proven in \cite[Thm. 6.2.1]{BVK}) we have a natural map $$\EExt^p (\LAlb(X),  \Z(1)) \to H^{p,1}(X)\cong H^{p-1}_\eh(X, \G_m)$$ which is an isomorphism, rationally, for all $p$ (see the motivic Albanese map displayed in \eqref{albmap} and \eqref{adjoint} below). We can also describe periods for 1-motives (see Definition~\ref{perhom}) in such a way that we obtain suitable Betti-de Rham realizations in period categories (see Definitions \ref{Def:drb} and \ref{Def:bdr}): a key point is that these realizations are fully faithful over  $\bar \Q$ (see Theorem~\ref{thm:ff1mot}).  The main ingredient in the proof of fullness is a theorem  due to Waldschmidt \cite[Thm. 5.2.1]{Wal} in transcendence  theory, generalizing the classical Schneider-Lang theorem (see also \cite[Thm. 4.2]{Bo}). An alternative proof can be given using a theorem of W\"ustholz \cite{WuC} (see our second proof of Theorem~\ref{thm:ff1mot}). A version of Baker's theorem and instances of Kontsevich period conjecture for 1-motives are further explored in a recent work of Huber and W\"ustholz \cite{HW}. Note that Kontsevich's period conjecture for 1-motives was formulated in \cite{WuL} (see also \cite[\S 23.3.3]{An}).  

Actually, we show  that the regulator  $r_\varpi^{p,1}$ can be revisited by making use of 1-motives (see Lemmas \ref{reproj} to \ref{regfact} and  Proposition~\ref{reg1}).
As a byproduct, all this promptly applies to show the surjectivity of $r_\varpi^{1,1}: H^{0}_\eh(X,\G_m)\to H^{1,1}_\varpi (X)$ via $\EExt (\LAlb(X),  \Z(1))$ verifying the conjecture for $p=1$ and $q=1$ (see Theorem \ref{per:1:1}). In fact, we can use the motivic Picard complex $\RPic (X)$ (see \cite[\S 8.3]{BVK})  so that $\EExt^p (\Z (0), \RPic (X))= \EExt^p  (\LAlb (X),  \Z(1))$ by Cartier duality showing that $\EExt (\LAlb(X),  \Z(1))$ is an extension of the finitely generated group $\Hom (\Z, \RA{p}(X))$ by a divisible group; we thus conclude that $r_\varpi^{p,1}$ induces a map 
$$\theta_\varpi^{p}: \Hom (\Z , \RA{p}(X))\to H^{p,1}_\varpi (X)$$
which in turn can be described making use of the mentioned Betti-de Rham realization. The surjectivity of $r_\varpi^{p,1}$ can be translated via $\theta_\varpi^{p}$ and the fullness of the Betti-de Rham realization.
For $p=1$, considering the 1-motive $\RA{1}(X) = [L_1^*\by{u_1^*} G_1^*]$ which is the Cartier dual of $\LA{1} (X)= [L_1\by{u_1} G_1]$, we get a canonical isomorphism
$$\ker u_1^*\cong \Hom (\Z , \RA{1}(X)) \cong H^1_\dr(X)\cap H^1(X_\an, \Z_\an (1))=H^{1,1}_\varpi (X).$$
In particular, we obtain that $H^1_\dr(X)\cap H^1(X_\an, \Z (1))=0$
if $X$ is proper. This vanishing for smooth projective varieties was previously obtained by Bost-Charles \cite[Thm. 4.2]{BC}. 

With some more efforts, making now use of the motivic complex $L\pi_0(X)$ along with its adjunction property (as stated in \cite[\S 5.4]{BVK}), we get a map $$\EExt^p (L\pi_0 (X),  \Z (0))\to H^{p,0}(X)\cong H^{p}_\eh(X,\Z).$$ Analysing the composition of this map for $p=1$ with $r_\varpi^{1,0}$ we see that $$r_\varpi^{1,0}: H^{1}_\et(X,\Z)\cong H^{1,0}_\varpi(X)= H^1_\dr(X)\cap H^1(X_\an, \Z_\an )$$ is an isomorphism (see Theorem \ref{per:1:0}), which yields the case $p=1$ and $q=0$  of our conjecture. In particular, 
$H^1_\dr(X)\cap H^1(X_\an, \Z_\an )=0$ 
for $X$ normal. This vanishing for smooth quasi-projective varieties was previously obtained by Bost-Charles \cite[Thm. 4.1]{BC}. 

For $p=1$ and $q\neq 0, 1$ we have that $H^{1,q}_\varpi (X)=0$ (see Corollary \ref{cor:per:1}) so that the period conjecture for motivic cohomology is trivially verified.

Remarkably, the description of the Grothendieck arithmetic invariants $H^{p,q}_\varpi (X)$ appears strongly related to the geometric properties encoded by motivic cohomology. These properties are almost hidden for smooth schemes, since the divisibility properties of motivic cohomology of $X$ smooth yields that for $p\notin [q, 2q]$ the surjectivity of $r_\varpi^{p,q}$ is equivalent to the vanishing $H^p_\dr(X)\cap H^p(X_\an, \Q_\an (q) )=0$. However, for $X$ smooth with a smooth compactification $\bar X$ and normal crossing boundary $Y$, we have that 
$$\ker\ (\Div^0_{Y}(\bar X)\longby{u_1^*} \Pic_{\bar X/\bar \Q}^0) \cong H^1_\dr(X)\cap H^1(X_\an, \Z_\an (1))$$
where $u_1^*$ is the canonical mapping sending a divisor $D$ supported on $Y$ to $\cO_{\bar X}(D)$.  In fact, here $\RA{1}(X)$ is Cartier dual of $\LA{1} (X) = [0\to \cA_{X/\bar \Q}^0]$, the Serre-Albanese semi-abelian variety (see \cite[Chap. 9]{BVK}). Therefore, there exist smooth schemes $X$ such that $H^{1,1}_\varpi (X)$ is non-zero and the vanishings in \cite[Thm. 4.1 \& 4.2]{BC} are particular instances of our descriptions.

With similar techniques one can make use of the Borel-Moore Albanese complex $\LAlb^c(X)$ (see \cite[Def. 8.7.1]{BVK}) to describe the compactly supported variant $H^{1,q}_{c,\varpi } (X)$, for any twist $q$. 

Finally, the cohomological Albanese complex $\LAlb^*(X)$ (see \cite[Def. 8.6.2]{BVK}) shall be providing a description of $H^{2d-j,q}_\varpi (X)$ for $d=\dim (X)$, at least for $j=0, 1$ and $q$ an arbitrary twist. 
A homological version of period regulators is also feasible and will be discussed in a future work.

\subsubsection*{Aknowledgements} We would like to thank Y. Andr\'e and J. Ayoub for some useful discussions on the matters treated in this paper. The first author was partially supported by  the Cariplo project n. 2017-1570 and all Italian authors acknowledge the support of the {\it Ministero dell'Istruzione,  dell'Universit\`a e della Ricerca} (MIUR)  through the Research Project  (PRIN 2010-11) ``Arithmetic Algebraic Geometry and Number Theory''.  We thank the referee for her/his many comments that helped us to considerably improve the exposition (starting with the suggestion to change the title to something reflecting more accurately the contents of the paper).

\section{Periods: constructions and conjectures} 
Let $\DM^\eff_\tau$ be the effective (unbounded) triangulated category of Voevodsky motivic complexes of $\tau$-sheaves over a field $K$ of zero characteristic, \ie  the full triangulated subcategory of $D(\Shv_\tau^\tr(\Sm_K))$ given by $\Aff^1$-local complexes  (\eg see \cite[\S 4.1]{AICM} and, for complexes bounded above, see also \cite[Lect. 14]{VL}). We here generically denote by $\tau$ either the Nisnevich or \'etale Grothendieck topology on $\Sm_K$, the category of smooth schemes which are of finite type over the field $K$. Let $\Z (q)$ for $q\geq 0$ be the Suslin-Voevodsky motivic complex regarded as a complex of \'etale sheaves with transfers. More precisely we consider a change of topology tensor functor $$\alpha : \DM^\eff_\Nis\to \DM^\eff_\et$$  and $\Z (q) = \alpha \Z_\Nis (q)$ (see \cite[Cor. 1.8.5 \& Def. 1.8.6]{BVK}) where $\Z_\Nis (q)$ is the usual complex for the Nisnevich topology (see also \cite[Def. 3.1]{SV}). We have the following canonical isomorphisms  $\Z (0) \cong \Z[0]$, $\Z (1)\cong \G_m [-1]$ and $\Z (q)\otimes \Z(q')\cong \Z (q+q')$ for any $q, q'\geq 0$ (see \cite[Lemma 3.2]{SV}). For any object $M\in \DM^\eff_\et$ we here denote $M(q)\df M\otimes \Z(q)$. Recall that by inverting the Tate twist $M\leadsto M(1)$ we obtain $\DM_\tau$ (where every compact object is isomorphic to $M(-n)$ for some $n\geq 0$ and $M$  compact and effective). For $M\in \DM^\eff_\et$ we shall define its motivic cohomology as
$$ H^{p,q}(M) \df \Hom_{\DM^\eff_\et} (M, \Z(q)[p]).$$
For any algebraic scheme $X$ we have the Voevodsky \'etale motive $M(X) = \alpha C_\d\Z_{tr}(X)\in \DM^\eff_\et$ where $C_\d$ is the Suslin complex and $\Z_{tr}(X)$ is the representable Nisnevich sheaf with transfers (see \cite[Def. 2.8, 2.14 \& Properties 14.5]{VL} and compare with \cite[Lemma 1.8.7 \& Sect. 8.1]{BVK}). We then write $H^{p,q}(X)\df H^{p,q}(M(X))$ and we refer to it as the \'etale motivic cohomology of $X$. We have an  isomorphism 
$$ H^{p,q}(X) \cong H^p_\eh(X, \Z(q))$$
where the last cohomology is, in general, computed by the $\eh$-topology (see \cite[\S 10.2]{BVK} and \cf \cite{AICM} and \cite[Prop. 1.8 \& Def. 3.1]{SV}). In particular, if $X$ is smooth $H^p_\eh(X, \Z(q))\cong H^p_\et(X, \Z(q))$. 

Note that we also have the triangulated category of motivic complexes without transfers $\DA^\eff_\et$ and if we are interested in rational coefficients we may forget transfers or keep the Nisnevich topology as we have equivalences $$\DA_{\et,\Q}^\eff\cong \DM_{\et,\Q}^\eff\cong \DM_{\Nis ,\Q}^\eff$$ (see \cite{AICM}, \cite[Cor. B.14]{AHop} and \cite[Thm. 14.30]{VL}). If we work with rational coefficients, we then have that motivic cohomology $H^{p,q}(X)_\Q$ is computed by the $\cdh$-topology, \ie $H^p_\eh(X, \Q(q))\cong H^p_\cdh(X, \Q(q))$, and $H^p_\cdh(X, \Q(q))\cong H^p_\Zar (X, \Q(q))$ if $X$ is smooth.

\subsection{de Rham regulator}
Denote by $\uOm$ the object of $\DM^\eff_\et$ which represents de Rham cohomology. More precisely we here denote $\uOm\df \alpha \uOm_\Nis$ where  $\uOm_\Nis$ is the corresponding object for the Nisnevich topology (see \cite[\S 2.1]{LW} and \cf \cite[\S 2.3]{AHop} without transfers). This latter $\uOm_\Nis$ is given by the complex of presheaves with transfers that associates to $X\in \Sm_K$ the global sections $\Gamma (X, \Omega^\d_{X/K})$ of the usual algebraic de Rham complex. 

For $M\in \DM^\eff_\et$ we shall denote (\cf \cite[\S 6]{LWdr} and \cite[Def. 2.1.1 \& Lemma 2.1.2]{LW})
$$ H^{p}_{\dr}(M) \df \Hom_{\DM^\eff_\et} (M, \uOm [p]).$$
For any algebraic scheme $X$ and $M=M(X)$ we here may also consider the sheafification of $\uOm$ for the $\eh$-topology.  Actually, we set $$H^p_\dr(X)\df  H^{p}_{\dr}(M(X))  \cong H^p_\eh(X, \uOm)$$
(see \cite[Prop. 10.2.3]{BVK}). Remark that this definition is equivalent to the definition of the algebraic de Rham cohomology in \cite[Chap. 3]{Hu} via the $h$-topology (as one can easily see via blow-up induction \cite[Lemma 10.3.1 b)]{BVK} after \cite[Prop. 3.2.4]{Hu} and \cite[Lemma 3.1.14]{Hu}).

 Note that for $q=0$ we have a canonical map $r^0:\Z(0)\to \uOm$ yielding a map
$$H^{p,0}(X) \cong  H^p_\eh(X,\Z)\to H^p_\dr(X) \cong H^p_\eh(X, \uOm).$$ 
For $q=1$ we have $r^1\df d\log : \Z (1) \to \uOm$ in  $\DM^\eff_\et$ (see \cite[Lemme 2.1.3]{LW} for the Nisnevich topology and apply $\alpha$) yielding a map
$$H^{p,1}(X) \cong  H^{p-1}_\eh(X,\G_m)\to H^p_\dr(X) \cong H^p_\eh(X, \uOm).$$  
Following \cite[(2.1.5)]{LW} an internal de Rham regulator $r^q$ in $\DM^\eff_\et$ for $q\geq 2$ is then obtained as the composition of
\begin{equation}\label{motreg}
r^q : \Z (q)\cong  \Z(1)^{\otimes q} \vlongby{d\log^{\otimes q}} \uOm^{\otimes q}\to \uOm .
\end{equation}
For $M\in \DM^\eff_\et$, composing a map $M\to \Z (q)[p]$ with $r^q[p]$ we  get an external {\it de Rham regulator} map 
\begin{equation}
{r^{p,q}_\dr}: H^{p,q}(M) \to H^p_\dr(M)
\end{equation}
and in particular for $M= M(X)$ we get
$${r^{p,q}_\dr}: H^{p,q}(X)\cong H^p_\eh(X, \Z(q)) \to H^p_\dr(X)\cong H^p_\eh(X, \uOm).$$
Note that if $X$ is smooth then $H^p_\eh(X, \uOm)\cong H^p_\et(X, \uOm)\cong H^p_\Zar (X, \Omega_X^\d)$ coincides with the classical algebraic de Rham cohomology (again, see \cite[Prop. 10.2.3]{BVK} and \cf \cite[Prop. 3.2.4]{Hu}) and we thus obtain  ${r^{p,q}_\dr}: H^p_\et(X, \Z(q)) \to H^p_\Zar (X, \Omega_X^\d)$ in this case.

\subsection{Periods}
As soon as we have an embedding $\sigma: K \into \C$ we may consider a Betti realization (\eg see \cite[\S 3.3]{LW} or \cite[Def. 2.1]{A}) in the derived category of abelian groups $D(\Z)$ as a triangulated functor
\begin{equation}\label{BettiR}
\beta_\sigma : \DM_\et^\eff\to D(\Z)
\end{equation}
such that $\beta_\sigma (\Z (q)) \cong \Z_\an(q)\df (2\pi i)^q\Z[0]$.  Actually, following Ayoub (see also \cite[\S 2.1.2]{AHop} and \cite[\S 1.1.2]{AHop2}) if we consider the analogue of the Voevodsky motivic category $\DM^\eff_\an$ obtained as the full subcategory of $D(\Shv_\an^\tr(\An_\C))$ given by $\Aff^1_\an$-local complexes, where we here replace smooth schemes $\Sm_K$ by the category $\An_\C$ of complex analytic manifolds, we get an equivalence 
$$\beta: \DM^\eff_\an\longby{\simeq} D(\Z)$$
such that $M_\an(X) \leadsto {\rm Sing}_*(X)$ is sent to the singular chain complex of $X\in \An_\C$. Moreover there is a natural triangulated functor
$$\sigma: \DM^\eff_\et\to \DM^\eff_\an$$
such that $M(X)\leadsto M_\an (X_\an)$ where the analytic space  $X_\an$ is given by the $\C$-points of the base change $X_\C$ of any algebraic scheme $X$. We then set $\beta_\sigma\df \beta\circ \sigma$. 
Thus it is clear that $\beta_\sigma (\Z[0] ) = \beta_\sigma (M(\Spec (K)) = \Z[0]$. Since a $K$-rational point of $X$ yields $M (X) = \Z\oplus \tilde M(X)$  we also see that $\beta_\sigma (\Z (1)[1]) \cong \beta_\sigma (\tilde M (\G_m))\cong \beta (\tilde M_\an(\C^*))\cong \Z_\an (1)[1]$ and then, $\beta_\sigma (\Z (q)) \cong \Z_\an(q)$ in general, as it follows from the compatibility of $\beta_\sigma$ with the tensor structures, \ie we here use the fact that $\beta_\sigma$ is unital and monoidal.
For $M\in \DM^\eff_\et$,  we denote
$$H^{p,q}_\an (M) \df \Hom_{D(\Z)}(\beta_\sigma M, \Z_\an(q)[p])$$
and we have a {\it Betti regulator} map 
\begin{equation}\label{Bettireg}
r^{p,q}_{\an}:H^{p,q} (M)\to H^{p,q}_\an (M)
\end{equation}
induced by $\beta_\sigma$. In particular, for $M= M(X)$, we obtain from Ayoub's construction (see also \cite[Prop. 4.2.7]{LW}):
\begin{lemma} 
For any algebraic $K$-scheme $X$  and any field homomorphism $\sigma : K \into \C$ we have
$$H^{p, q}_\an (X)\df \Hom_{D(\Z)}(\beta_\sigma M(X), \Z_\an(q)[p])\cong H^p(X_\an, \Z_\an (q))$$
and a Betti regulator map 
$$ r_\an^{p,q}:H^{p,q}(X) \to H^p(X_\an, \Z_\an (q)).$$
\end{lemma}

Recall that the functor $\beta_\sigma$ admits a right adjoint $\beta^\sigma : D(\Z) \to \DM^\eff_\et$ (see \cite[Def. 1.7]{AHop2}). Note that the Betti regulator \eqref{Bettireg}
is just given by composition with the unit 
\begin{equation}\label{motregsigma}
r_\sigma^q: \Z(q)\to \beta^\sigma\beta_\sigma (\Z (q))
\end{equation}
 of the adjunction. 
Actually, by making use of the classical Poincar\'e Lemma and Grothendieck comparison theorem (\cite[Thm. $1^\prime$]{GrdR}) we get: 
\begin{lemma}[\protect{Ayoub}] \label{APLemma}
There is a canonical quasi-isomorphism
$$\varpi^{q}: \beta^\sigma\beta_\sigma (\Z (q))\otimes_\Z \C \longby{\qi}\uOm \otimes_K \C$$
whose composition with $r_\sigma^q$ in \eqref{motregsigma} is the regulator $r^q$ in \eqref{motreg} after tensoring with $\C$.
\end{lemma}
\begin{proof} See \cite[Cor. 2.89 \& Prop. 2.92]{AHop} and also \cite[\S 3.5]{APKZ}. \end{proof} 
\begin{remark}
Note that applying $\beta_\sigma$ to $\varpi^{q}$ we obtain a quasi-isomorphism $\beta_\sigma (\varpi^{q})$ such that
 \[\xymatrix{\C \cong \beta_\sigma (\Z (q))\otimes_\Z\C\ar@/_1.6pc/[rrrr]_-{\beta_\sigma (r^q)_\C}\ar[rr]_{\beta_\sigma (r_\sigma^{q})_\C} && \ar@/_1.2pc/[ll]\beta_\sigma\beta^\sigma\beta_\sigma (\Z (q))\otimes_\Z \C \ar[rr]^{\ \ \beta_\sigma (\varpi^{q})} && \beta_\sigma (\uOm )\otimes_K \C}\]
where $\beta_\sigma (r^q)_\C$ is a split injection but it is not a quasi-isomorphism (\cf \cite[\S 4.1]{LW}). 
\end{remark}
For $M\in \DM^\eff_\et$, by composition with $\varpi^q$ we get a period isomorphism 
$$\varpi^{p,q}_M: H^{p,q}_\an (M)\otimes_\Z \C \longby{\simeq} H^p_\dr(M)\otimes_K\C .$$
\begin{defn}\label{classper} For any scheme $X$ we shall call {\it period isomorphism} the $\C$-isomorphism
$$\varpi^{p,q}_X: H^p(X_\an, \Z_\an (q))\otimes_\Z \C \longby{\simeq}  H^p_\dr(X)\otimes_K\C$$
obtained by setting $\varpi^{p,q}_X\df \varpi^{p,q}_{M(X)}$ as above.
We shall denote $\eta^{p,q}_X\df (\varpi^{p,q}_X)^{-1}$ the inverse of the period isomorphism.
\end{defn}
We also get the following compatibility.
\begin{propose} \label{periodsquare}
For  $M\in \DM^\eff_\et$ along with a fixed embedding $\sigma: K \into \C$ the inverse of the period isomorphism $\varpi^{p,q}_M$ above induces a commutative diagram
 \[\xymatrixcolsep{2pc}\xymatrix{H^{p,q}(M) \ar[rr]^{r_\an^{p,q}} \ar[d]_{r^{p,q}_\dr}&& H^{p,q}_\an (M)\ar[dr]^{\iota^{p,q}_\an}&\\
H^p_\dr(M)\ar[rr]^{}\ar@/_1.6pc/[rrr]_{\iota^{p,q}_\dr}&& H^p_\dr(M)\otimes_K \C \ar[r]^{\simeq} &  H^{p,q}_\an (M)\otimes_\Z \C 
 }\]
where $\iota^{p,q}_\dr$ and $\iota^{p,q}_\an$ are the canonical mappings given by tensoring with $\C$.   \end{propose}
 \begin{proof} This easily follows from Lemma \ref{APLemma}. In fact, by construction, the claimed commutative diagram can be translated into the following commutative square:
\[\xymatrix{\Hom_{\DM^\eff} (M, \Z(q)[p]) \ar[d]_{r_\dr^{p,q}\df r^q[p]\circ -}\ar[rr]^-{\iota^{p,q}_\an\circ\, r_\an^{p,q}} & & \ar[d]_{\varpi^q[p]\circ -}\Hom_{\DM^\eff}( M,\beta^\sigma\beta_\sigma\Z(q)[p])_\C\\
 \Hom_{\DM^\eff} (M, \uOm [p])\ar@/^1.2pc/[rru]_{\iota^{p,q}_\dr}\ar[rr]_-{-\otimes_K\C} & & \Hom_{\DM^\eff}(M, \uOm [p])_\C.}\]
\end{proof}
\begin{cor}\label{omegacom}  Let $X$ be an algebraic $K$-scheme along with a fixed embedding $\sigma: K \into \C$. The period isomorphism $\varpi^{p,q}_X$ above induces a commutative square
 \[\xymatrixcolsep{2pc}\xymatrix{ H^{p,q}(X) \ar[rr]^{r_\an^{p,q}} \ar[d]_{r^{p,q}_\dr}&& H^p(X_\an, \Z_\an (q))\ar[d]^{\iota^{p,q}_\an}\\
H^p_\dr(X)\ar[rr]^{\iota^{p,q}_\dr}&& H^p(X_\an, \C).
 }\]
 \end{cor}
Note that from Corollary \ref{omegacom} we get a refinement of the Betti regulator.
\begin{defn} \label{defper}
Define the {\it algebraic}\, singular cohomology classes as the elements of the subgroup $H^{p,q}_{\rm alg}(X)\df \im  r_\an^{p,q}\subseteq H^p(X_\an, \Z_\an (q))$ given by the image of the motivic cohomology under the Betti regulator $r_\an^{p,q}$. 

Define the {\it $\varpi$-algebraic}\, singular cohomology classes by the subgroup 
$$H^{p,q}_\varpi(X) \df H^p_\dr(X)\cap H^p(X_\an, \Z_\an (q))\subseteq H^p(X_\an, \Z_\an (q))$$ where $\cap$ means that we take elements in $H^p(X_\an, \Z_\an (q))$ which are given by the inverse image (under $\iota^{p,q}_\an$) of elements in $H^p_\dr(X)$ regarded (under $\iota^{p,q}_\dr$) inside $H^p(X_\an, \C)$  via the isomorphism $\varpi^{p,q}_X$ above. 

The groups $H^{p,q}_\varpi (X)$ shall be called {\it period cohomology} groups and 
$$ r_\varpi^{p,q}:H^{p,q}(X) \to H^{p,q}_\varpi (X)$$ 
induced by $r_\dr^{p,q}$ and $r_\an^{p,q}$ shall be called the {\it period}\ regulator. 
\end{defn}
We get that:
\begin{cor} $H^{p,q}_{\rm alg}(X)\subseteq H^{p,q}_\varpi(X)$. \end{cor}

For example, all torsion cohomology classes are $\varpi$-algebraic: we shall see in Lemma \ref{torhyp} that they are also algebraic. 

In particular, if $H^p(X_\an, \Z_\an (q))$ is all algebraic, \ie the Betti regulator $r_\an^{p,q}$ is surjective, then the canonical embedding $\iota^{p,q}_\an$ of singular cohomology $H^p(X_\an, \Q_\an (q))$ in the $\C$-vector space $H^p(X_\an, \C)$ factors through an embedding of $H^p(X_\an, \Q_\an (q))$ into the $K$-vector space $H^p_\dr(X)$. If $K=\bar \Q$ this  rarely happens. For example, if $p=0$ it happens only if $q=0$ and in this case $r_\varpi^{0,q}$ is always surjective (as $H^{0,q}_\varpi (X)=0$ for $q\neq 0$). 

\subsection{Period conjecture for motivic cohomology} 
Over $K =\overline{\Q}$ it seems reasonable to make the conjecture that all $\varpi$-algebraic classes are algebraic, \ie to conjecture that the period regulator $r_\varpi^{p,q}$ is surjective. In other words we may say that the period conjecture for motivic cohomology holds for $X$, in degree $p$ and twist $q$, if
\begin{equation}\label{periodconj}
H^{p,q}_{\rm alg}(X) = H^{p,q}_\varpi(X).
\end{equation}  
Over a number field we may expect that this holds rationally. 
If \eqref{periodconj} holds we also have that $H^p(X_\an, \Z_\an (q))$ modulo torsion embeds into $H^p_\dr(X)$ if and only if $H^p(X_\an, \Z_\an (q))$ is all algebraic. 
Note that using Proposition \ref{periodsquare} we can define $H^{p,q}_\varpi (M)$ providing a version of the period conjecture for any object $M\in \DM^\eff_\et$.
\begin{propose} \label{zerotwist}
For any $q\geq 0$ the period conjecture \eqref{periodconj} holds true for $X$, in degree $p$ and twist $r$, if and only if it holds true for $M(X)(q)$, in degree $p$ and twist $q+r$.
\end{propose}
\begin{proof} By Voevodsky cancellation theorem \cite{VC} we have that twisting by $q$ in motivic cohomology $H^{p,r} (X)\longby{\simeq} H^{p,q+r} (M(X)(q))$ is an isomorphism of groups. If $M = M(X)(q)$ with $q\geq 0$ we also get $H^{p,r}_\varpi (X)\longby{\simeq} H^{p,q+r}_\varpi (M(X)(q))$ canonically by twisting. In fact, we have a diagram induced by twisting 
 \[\xymatrix{
H^p(X_\an, \Z_\an(r) )\otimes \C\ar[r]^-{q_\C}_-{\tilde{ }}\ar[d]|*=0[@]{\tilde{ }}_{\omega_X^{p,r}}& H^{p,q+r}_\an (M (X)(q))\otimes_\Z \C\ar[d]|*=0[@]{\tilde{ }}_{\varpi^{p,q+r}_{M (X)(q)}}\\
H^p_\dr(X)\otimes_K\C\ar@{.>}[r]^-{q_\dr}_-{\tilde{ }}& H^p_\dr(M (X)(q))\otimes_K\C}\]
where $q_\C\df q\otimes \C$ is the $\C$-isomorphism given by the canonical integrally defined mapping $q \colon H^p(X_\an, \Z_\an(r) )\by{\simeq} H^{p,q+r}_\an (M (X)(q))$ which is sending a $p$-th cohomology class regarded as a map $\beta_\sigma M(X)= {\rm Sing}_*(X_\an) \to \Z_\an (r)[p]$ in $D (\Z)$ to the $q$-twist $\beta_\sigma M(X)(q)= {\rm Sing}_*(X_\an)(q) \to \Z_\an(q+r)[p]$. Similarly,  the $\C$-isomorphism $q_\dr$ is induced by twisting, since $\uOm (-q) \longby{\qi} \uOm$ is a canonical isomorphism in $\DM^\eff_\et$ and the claim follows.
\end{proof}
For $X$ smooth we have that $H^{p,q}(X)\cong H^p_\et (X, \Z (q))$ and with rational coefficients we have that $H^p_\et (X, \Q (q)) \cong CH^q(X,2q-p)_{\Q}$. In particular, if $X$ is smooth and $p=2q$ we get that $r_\varpi^{2q,q}$ is the modern refinement of the classical cycle class map with rational coefficients \begin{equation} \label{cycle}
 r_\varpi^{2q,q}=c\ell^q_\varpi: CH^q(X)_{\Q}\to H^{2q,q}_\varpi(X)_{\Q}
\end{equation}
for codimension $q$ cycles on $X$ considered in \cite{BC}. In this case, the period conjecture \eqref{periodconj} with rational coefficients coincides with the classical Grothendieck period conjecture for algebraic cycles: see \cite[\S 1.1.3]{BC} and \cite[Prop. 2.13-14]{BC} comparing it with the conjecture on torsors of periods. 
\begin{remark}
For $K=\C$ we may also think to refine the Hodge conjecture as previously hinted by Beilinson, conjecturing the surjectivity of
$$r^{p, q}_{\rm Hodge}: H^{p,q}(X)_\Q \to \Hom_{\MHS}(\Q(0), H^{p}(X)(q)).$$  
However,  such a generalization  doesn't hold, in general, \eg see \cite{JL}. 
\end{remark}

\subsection{Torsion cohomology classes are algebraic}
Consider $\Z/n (q)\df \Z (q)\otimes \Z/n$. By Suslin-Voevodsky rigidity we have a quasi-isomorphism of complexes of \'etale sheaves $\mu_n^{\otimes q}\to \Z/n (q)$  yielding $H^p_\eh(X,\Z/n (q))\cong H^p_\et(X,\mu_n^{\otimes q})\cong H^p_\eh(X,\mu_n^{\otimes q})$. For a proof of this key result see \cite[Thm. 10.2 \& Prop. 10.7]{VL} for $X$ smooth and make use of \cite[Prop. 12.1.1]{BVK} to get it in general.
\begin{lemma}\label{rig} For any algebraic scheme $X$ over $K =\overline{K}\into \C$ we have
$H^p_\eh(X,\Z/n (q))\cong H^p(X_\an,\Z/n)$. 
\end{lemma}
\begin{proof}  As  \'etale cohomology of $\mu_n^{\otimes q}$ is invariant under the extension $\sigma: K\into \C$ of algebraically closed fields we obtain the claimed comparison from the classical comparison result after choosing a root of unity.
\end{proof}
We then have (\cf \cite[Prop. 3.1]{RS}):
\begin{lemma} \label{torhyp}
The regulator $r_\varpi^{p,q}\mid_\tor:H^{p,q}(X)_\tor \onto H^{p,q}_\varpi (X)_\tor$ is surjective on torsion and $r_\varpi^{p,q}\otimes\Q/\Z :H^{p,q}(X)\otimes\Q/\Z \into H^{p,q}_\varpi (X)\otimes\Q/\Z$ is injective.
\end{lemma}
\begin{proof}
By construction, for any positive integer $n$, comparing the usual universal coefficient exact sequences, we have the following commutative diagram with exact rows 
  \[\xymatrix{0\to  H^p_\eh(X,\Z (q))/n \ar[r]\ar[d]^{r_\an^{p,q}/n}  & H^p_\eh(X,\Z/n (q))\ar[r]\ar@{=}[d]^{\ref{rig}}& { }_{n}H^{p+1}_\eh(X,\Z (q))\to 0\ar[d]^{{ }_{n}r_\an^{p+1,q}}\\
 0\to  H^p(X_\an,\Z_\an)/n \ar[r] &   H^p(X_\an,\Z/n)\ar[r]& { }_{n}H^{p+1}(X_\an,\Z_\an)\to 0.
  }\]
Passing to the direct limit on $n$ we easily get the claim. In fact, ${ }_{n}H^{p,q}_\varpi (X) = { }_{n}H^{p}(X_\an,\Z_\an)$ and $r_\an^{p,q}/n$ factors through $r_\varpi^{p,q}/n$. 
\end{proof} 
\begin{lemma} We have that $r_\varpi^{p,q}\otimes\Q$ is surjective if and only if $r_\varpi^{p,q}$ is surjective; moreover, if this is the case $r_\varpi^{p,q}\otimes\Q/\Z$ is an isomorphism. \end{lemma}
\begin{proof}
This follows from a simple diagram chase.
\end{proof}
In the situation that $H^{p,q}(X)\otimes\Q/\Z = 0$ the period conjecture for motivic cohomology \eqref{periodconj} is then equivalent to  \begin{equation}\label{vanishing}H^p_\dr(X)\cap H^p(X_\an, \Q_\an (q))=0.\end{equation}

In particular: 
\begin{propose}\label{perconjsm}
 If $X$ is smooth then \eqref{periodconj} for $p\notin [q, 2q]$ is equivalent to \eqref{vanishing}. If $X$ is smooth and proper then \eqref{periodconj} is equivalent to the surjectivity of $c\ell_\varpi^q$ in \eqref{cycle} for $p=2q$ and 
to the vanishing \eqref{vanishing} for $p\neq 2q$.
\end{propose}
\begin{proof} In fact, by the Appendix \ref{KD}, Theorem \ref{t2}, we have that for $p\notin [q, 2q]$ the group $H^{p,q}(X)$ is an extension of torsion by divisible groups so that $H^{p,q}(X)\otimes\Q/\Z = 0$. If $X$ is proper the latter vanishing holds true for all $p\neq 2q$.
\end{proof}

Proposition \ref{perconjsm} explains some weight properties related to the Grothendieck period conjecture, weight arguments which are also considered in \cite{BC}. 
\begin{remark} For $K= \C$ we have that $r_\beta^{p,q}\mid_\tor:H^{p,q}(X)_\tor \onto H^{p} (X_\an, \Z_\an (q))_\tor$ is surjective (as also remarked in \cite{RS} for $X$ smooth projective): torsion motivic cohomology classes supply the defect of algebraic cycles providing the missing torsion algebraic cycles. In fact, from the well known Atiyah-Hirzebruch-Totaro counterexamples to the integral Hodge conjecture we know that $c\ell^p : CH^p(X) \to H^{2p} (X_\an, \Z_\an (p))$ cannot be surjective on torsion for $p\geq 2$ in general. 
\end{remark}

\section{Periods of 1-motives: fullness of Betti-de Rham realizations}
Let $\tM (K)$ be the abelian category of $1$-motives with torsion over $K$ (see \cite[App. C]{BVK}). We shall drop the reference to $K$ if it is clear from the context. We shall denote $${\sf M}_K=[{\sf u}_{K}\colon {\sf L}_K\to {\sf G}_K]\in \tM (K)$$ a  $1$-motive with torsion with ${\sf L}_K$ in degree $0$ and ${\sf G}_K$ in degree $1$; for brevity, we shall write ${\sf M}_K={\sf L}_K[0]$ if ${\sf G}_K=0$ and ${\sf M}_K={\sf G}_K[-1]$ if ${\sf L}_K=0$ and we omit the reference to $K$ if unnecessary. Let ${\sf M}_\tor := [{\sf L}_\tor \cap \ker({\sf u}) \to 0]$ be the torsion part of ${\sf M}_{K}$,  let ${\sf M}_\fr:=[{\sf L} /{\sf L}_\tor\to {\sf G}/{\sf u}({\sf L}_\tor )]$ be the free part of  ${\sf M}_{K}$, and let  ${\sf M}_\tf := [{\sf L}/ {\sf L}_\tor\cap \ker({\sf u}) \to {\sf G} ]$ be the torsion free part of ${\sf M}_{K}$.
There are  short exact sequences of complexes
\begin{eqnarray}\label{eq.mtf}
0\to {\sf M}_\tor\to {\sf M}_{K} \to {\sf M}_\tf\to 0
\end{eqnarray}
and
\begin{eqnarray}
\label{eq.mtf2}
0\to [F=F]\to {\sf M}_\tf \to {\sf M}_\fr\to 0,
\end{eqnarray}
where $F={\sf L}_\tor/{\sf L}_\tor\cap \ker({\sf u}) $.
Let ${\sf M}_\ab$ denote the $1$-motive with torsion $[{\sf L}\to {\sf G}/{\sf T}]$ where ${\sf T}$ is the maximal subtorus of ${\sf G}$. Recall (see \cite[Prop. C.7.1]{BVK}) that the canonical functor $\M \to \tM$ from Deligne 1-motives admits a left adjoint/left inverse given by ${\sf M}\leadsto {\sf M}_\fr$. 

Any $1$-motive ${\sf M}=[  {\sf L} \to {\sf G}] $ is canonically endowed with an increasing filtration of sub-$1$-motives, the {\it weight filtration},   defined as follows:
\begin{eqnarray}
\label{eq.weight} W_i({\sf M})=\left\{ \begin{matrix} {\sf M}&\text{if }&i\geq 0\\
{\sf G}[-1]&\text{if } &i=-1\\
{\sf T}[-1]&\text{if } &i=-2\\
0&\text{if } &i\leq -3 \\
\end{matrix}\right  .
\end{eqnarray}
with $\sf T$ the maximal subtorus of $\sf G$.
We have that $D^b(\tM )\cong D^b(\M )$ (see \cite[Thm. 1.11.1]{BVK}) and that there is a canonical embedding (see \cite[Def 2.7.1]{BVK}) 
\begin{equation}\label{tot}
\Tot : D^b(\M)\into \DM^\eff_\et
\end{equation}
so that we can also regard $1$-motives as motivic complexes of \'etale sheaves. The restriction of the Betti realization $\beta_\sigma$ in \eqref{BettiR} can be described explicitly for 1-motives via Deligne's Hodge realization (see \cite[Thm. 15.4.1]{BVK}). Similarly, the restriction of the de Rham realization in \cite{LW} can be described via Deligne's de Rham realization as follows.

\subsection{de Rham realization} 
Let $K$ be a field of characteristic zero and let ${\sf M}_K=[{\sf u}_{K}\colon {\sf L}_K\to {\sf G}_K]\in \tM (K)$ be  a  $1$-motive  with torsion over $K$.  Note that for  ${\sf M}_{K}^\natural \df [{\sf u}^\natural_K : {\sf L}_K\to {\sf G}^\natural_K]$ the universal $\G_a$-extension of ${\sf M}_{K}$ we have
$$0\to \V ({\sf M}) \to {\sf M}_{K}^\natural\longby{\rho^{\sf M}} {\sf M}_{K}\to 0$$
where $\V ({\sf M})\df \Ext ({\sf M}_{K},\G_a)^\vee$.  The existence of universal extensions is well-known when ${\sf L}_{K}$ is torsion-free; for the general case see \cite[Proposition 2.2.1]{BVB}. Recall (see \cite[\S 10.1.7]{De}) the following

\begin{defn} The de Rham realization of ${\sf M}_K$ is $$T_\dr({\sf M}_K)\df \Lie ({\sf G}^\natural_K)$$ as a $K$-vector space. 
\end{defn}

\begin{remark}\label{hodgefilt}
Note that $\rho^{\sf M} = (id_{\sf L}, \rho_{\sf G})$ where $\rho_{\sf G}\colon {\sf G}^\natural_K\onto {\sf G}_K$ is a quotient and  $\ker \rho_{\sf G} = \V ({\sf M})$ so that ${\sf G}_K$ is the semiabelian quotient of ${\sf G}^\natural_K$ and ${\sf u}^\natural_K$ is a canonical lifting of ${\sf u}_K$, \ie ${\sf u}_K=\rho_{\sf G}\circ{\sf u}_K^\natural$.  Further $\V({\sf M})\subseteq T_\dr({\sf M}_K)$ is also the kernel of the morphism $$d\rho_{\sf G}\colon \Lie({\sf G}^\natural_{K})\to \Lie({\sf G}_{K})$$ induced by $\rho_{\sf G}$, so that $T_\dr({\sf M}_{K})$ together with the $K$-subspace $\V({\sf M})$ can be regarded as a filtered $K$-vector space. This datum is called the {\it Hodge filtration} of $T_\dr({\sf M}_{K})$.
\end{remark}

The algebraic $K$-group ${\sf G}^\natural_K$ fits in the following diagram \cite[(2.15)]{Ber}
\begin{eqnarray*}
	\xymatrix{
		0\ar[r]&\V ({\sf A}) \ar[r]\ar@{^{(}->}[d]^i &{\sf A}^\natural\times_{\sf A} {\sf G} \ar[r]  \ar@{^{(}->}[d] &
	{\sf G}\ar[r]\ar@2{-}[d] &0\\
		0\ar[r]& \V ({\sf M}) \ar[r]\ar@{->>}[d]& {\sf G}^\natural
		\ar[r]^{\rho_{{\sf G}}}\ar@{->>}[d] &
	{\sf G} \ar[r] &0\\
		& \V({\sf L})   \ar@2{-}[r] &  {\sf L} \otimes \G_a  & &
}\end{eqnarray*}
where   we have omitted subscripts $K$ and written $\V({\sf A}))$ for $\V({\sf A}[-1])= \V({\sf G}[-1])$.
\begin{lemma}\label{baseC}
For $K \subset K'$ we have a natural isomorphism
\[({\sf M}^\natural_K)_{K'} \cong ({\sf M}_{K'})^\natural. 
\]
\end{lemma}

\subsection{Base change to $\C$ and periods} 
Consider $K$ a subfield of $\C$ and let ${\sf M}_\C=[{\sf u}_{\C}\colon {\sf L}_\C\to {\sf G}_\C]$ be the base change of ${\sf M}_K$ to $\C$. Let $T_\Z({\sf M}_\C)$ be the finitely generated abelian group in the usual Deligne-Hodge
realization of ${\sf M}_\C$ (see \cite[10.1.3]{De} and \cite[\S 1]{BRS}) given by the pull-back
\[
\xymatrix{0\ar[r]& H_1({\sf G}_\C )  \ar[r]  \ar@{=}[d]& \Lie({\sf G}_{\C})\ar[r]^{\exp}&{\sf G}_{\C}\ar[r]&0\\
0\ar[r]&  T_\Z({\sf G}_\C)  \ar[r] \ar@{=}[u]&   T_\Z({\sf M}_\C)\ar[r]^{\tilde \exp}\ar[u]^{\tilde{\sf u}_\C} &{\sf L}_\C\ar[r]\ar[u]_{\sf u_\C}&0}
\]
where for brevity  $T_\Z({\sf G}_\C)$ denotes $T_\Z({\sf G}_\C[-1])$ which by definition is $ H_1({\sf G}_\C ) $.
After base change to $\C$ and Lemma \ref{baseC} we then get $({\sf M}^\natural_K)_{\C} \cong ({\sf M}_{\C})^\natural$ hence  an isomorphism 
\begin{equation}\label{iota}
\iota \colon T_\dr({\sf M_\C})\longby{\simeq} T_\dr({\sf M}_K)\otimes_K\C
\end{equation} and a commutative diagram
\begin{equation}\label{4sq}
\xymatrix{0\ar[r]& H_1({\sf G}_\C )  \ar[r]  \ar@{=}[d]& \Lie({\sf G}_{\C})\ar[r]^{\exp}&{\sf G}_{\C}\ar[r]&0\\
0\ar[r]&  H_1({\sf G}_\C^\natural )  \ar[r]  & \Lie({\sf G}_{\C}^\natural)\ar[r]^{\exp}\ar[u]_{d \rho_{\sf G}}&{\sf G}_{\C}^\natural\ar[r]\ar[u]^{\rho_{\sf G}}&0 \\
0\ar[r]&  T_\Z({\sf G}_\C)  \ar[r] \ar@{=}[u] &   T_\Z({\sf M}_\C)\ar@/^2.7pc/[uu]^(0.3){\tilde{\sf u}_\C}\ar[r]^{\tilde \exp}\ar@{.>}[u] &{\sf L}_\C\ar[r]\ar[u]^{\sf u^\natural_\C}\ar@/_2.0pc/[uu]_(0.3){{\sf u}_\C} &0
  }\end{equation}
where the dotted arrow exists by definition of $T_\Z({\sf M}_\C) $ and the fact that the upper right-hand square is cartesian. Hence also the lower  right-hand square is cartesian and the sequence on the bottom is equivalently obtained by pull-back of the upper sequence via ${\sf u}_{\C}$ or of the sequence in the middle via ${\sf u^\natural_\C} $. Further $T_\Z({\sf G}_\C)$ is identified with the kernel of both exponential maps and the dotted arrow gives a homomorphism $ T_\Z({\sf M}_\C)\to  T_\dr({\sf M}_\C):= \Lie({\sf G}_{\C}^\natural)$. 
Finally note that the weight filtration in \eqref{eq.weight} gives a filtration on  $T_\Z({\sf M}_\C)$: the immersion ${\sf T}_\C\to {\sf G}_\C$ gives an inclusion $  T_\Z({\sf T}_\C)=H_1({\sf T}_\C) \subseteq  H_1({\sf G}_\C)=T_\Z({\sf G}_\C)$ while    $  T_\Z({\sf G}_\C) \subseteq  T_\Z({\sf M}_\C)$ comes from the previous diagram. 
\begin{defn} \label{perhom}
The homomorphism of {\it periods}\, is the unique homomorphism 
$$\varpi_{\sf M,\Z}: T_\Z({\sf M}_\C)\to T_\dr({\sf M}_K)\otimes_K\C$$ 
that yields $d \rho_{\sf G}\circ \varpi_{\sf M,\Z} = \tilde{\sf u}_\C$ and $\exp \circ\varpi_{\sf M,\Z}= {\sf u}^\natural_\C \circ\tilde{\exp}$ under the identification given by the isomorphism $\iota$ in \eqref{iota}. 
\end{defn}

Note that $\tilde{\sf u}_\C$ is the pull-back of ${\sf u}_\C$ along $\exp$ and for $x\in {\sf L}_\C$ we may pick $\tilde{\log} (x)\in T_\Z({\sf M}_\C)$, \ie such that $\tilde{\exp}(\tilde{\log}(x))=x$. 
We then get
\begin{equation}\label{logper} 
{\sf u}_\C (x)= \exp( \tilde{\sf u}_\C (\tilde{\log}(x)))= \exp (d\rho_{\sf G}(\varpi_{{\sf M},\Z}(\tilde{\log}(x)))).
\end{equation}

\begin{thm} \label{BettideRham}
The induced $\C$-linear mapping 
\[\varpi_{\sf M,\C}: T_\C({\sf M}_\C)\df T_\Z({\sf M}_\C)\otimes_\Z\C \to T_\dr({\sf M}_K)\otimes_K\C\]
is an isomorphism.
\end{thm}
\begin{proof} Making use of the identification in \eqref{iota} we are left to see that it holds true for $K=\C$. The case of $L$ without torsion is treated by Deligne \cite[10.1.8]{De}. Actually, an easy proof can be given by d\'evissage to the case of lattices, tori and abelian varieties. For the general case note that $\varpi_{\sf M,\C}=\varpi_{{\sf M}_\tf,\C}$  by \eqref{eq.mtf}. Indeed $ T_\dr({\sf M}_\tor)=0$ and the kernel of the canonical morphism $T_\Z({\sf M}_\C)\to T_\Z({\sf M}_{\tf,\C})$ is torsion. Further by \eqref{eq.mtf2} the map $T_\Z({\sf M}_{\tf,\C})\to T_\Z({\sf M}_{\fr,\C})$ is an isomorphism and we have an exact sequence $$ 0\to [F=F]\to {\sf M}_\tf^\natural \to {\sf M}_\fr^\natural\to 0$$ so that  the canonical morphism $T_\dr({\sf M}_\tf)\to T_\dr({\sf M}_\fr)$ is an isomorphism too. Hence $\varpi_{{\sf M}_\tf,\C}=\varpi_{{\sf M}_\fr,\C}$. We conclude that $\varpi_{\sf M,\C}=\varpi_{{\sf M}_\fr,\C}$ and the latter is an isomorphism since ${\sf M}_\fr$ is a Deligne $1$-motive.  
\end{proof} 

\begin{examples}\label{e.2} $\bullet$ If ${\sf M}_{K}=[0\to \G_m]$, then $T_\Z({\sf M}_{\C})=\Z$ and the first and second rows in   \eqref{4sq} are  given by  $0\to \Z\by{2\pi i} \C\by{\exp} \C^*\to 0$. Hence   $\varpi_{{\sf M},\Z}=\tilde{\sf u}_\C\colon \Z\to \C, x\mapsto 2\pi i x$ and $\varpi_{{\sf M},\C} \colon \C\to \C, z\mapsto 2\pi i z$. 

$\bullet$ If  ${\sf M}_{K}=[{\sf L}_{K}\to 0]$, then $T_{\Z}({\sf M}_{\C})={\sf L}_{\C}$ and $T_\dr({\sf M}_{\C} )={\sf L}_{\C}\otimes_{\Z} \C$ and the map $\varpi_{{\sf M},\Z}   $ is the homomorphism $ {\sf L}_\C\to {\sf L}_{\C}\otimes_{\Z} \C,  x\mapsto x\otimes 1$ and $\varpi_{{\sf M},\C}$ is the identity map.

$\bullet$ If ${\sf M}_{K}=[{\bf u}\colon \Z\to \G_m]$ with ${\sf u}(1)=a\in K^*$, then ${\sf G}^\natural =\G_m\times \G_a$. Once fixed a complex logarithm $\log a$ of $a$ we can construct an isomorphism $\Z\times \Z\to T_\Z({\sf M}_\C)=\{(z,y)\in \C\times \Z| \exp(z)=a^y\}$ that maps the pair $(k,y)$ to  $(y\log a+2\pi i k, y)$. The homomorphism of periods becomes then the map $\varpi_{\sf M,\Z}\colon \Z\times  \Z\to \C\times  \C$ that sends $ (k, y)$ to $( y\log a+ 2\pi i k,y )$.
\end{examples}

Note that over $\C$ the  Hodge  filtration $\mathbb V( {\sf M})\subseteq T_\dr({\sf M}_K )$ of Remark \ref{hodgefilt} is obtained from the Hodge filtration of  $ T_\C({\sf M}_\C )$ via $ \varpi_{{\sf M},\C}$.

\subsection{Periods and transcendence}\label{sec:transresults}
The proof of Theorem \ref{thm:ff1mot}, which is the main outcome of the second section of this paper, makes use of deep results of transcendence theory that we recall below.
First consider  \cite[Theorem 2]{La}.

\begin{thm}\label{lang}
Let ${\sf A}_K$ be an abelian variety of dimension $d$ over $K=\bar \Q$. Let $\Theta\colon\C^d\to {\sf A}_\an$ be the homomorphism given by the theta functions, inducing an isomorphism of the complex
torus onto ${\sf A}_\an$. Assume that the derivations $\partial/\partial z_i$, $(i = 1, \dots, d)$ are defined over $K$. If $\alpha=(\alpha_i)\in \C^d$
is a complex vector $\neq  0$ such that all $\alpha_i$ lie in $K$, then $\Theta(\alpha)$ is transcendental over $K$. In particular, 	the periods are transcendental.
\end{thm}
It can be generalized to semiabelian varieties as follows.

\begin{thm}\label{trasc}
	Let ${\sf G}_K$ be a semiabelian variety over $K=\bar \Q$. If $0\neq x\in \Lie({\sf G}_{K})$, then $\exp(x)\in {\sf G}_{\C}(\C)$ is transcendental over $K$. In particular,
	$\Lie({\sf G}_{K})\cap \ker(\exp)=\{0\}$ in $\Lie({\sf G}_{\C})$.
\end{thm} 
\begin{proof}
The assertion  is known if ${\sf G}_{K}=\G_{m,K}^d$ due to fundamental work of Hermite and Lindemann on the transcendence of $e^\beta$ for $\beta$ a non zero algebraic number.
The case  ${\sf G}_K$ an abelian variety is Theorem \ref{lang}.
The general case follows then by d\'evissage. 
\end{proof}  

This type of results has been further generalised by  Waldschmidt (\cite[Thm. 5.2.1]{Wal}):

\begin{thm}\label{thm:Wal}
	Let ${\sf G}$ be a commutative connected algebraic group over $K=\bar \Q$. Let $\varphi\colon \C^n \to {\sf G}_\an$ be an analytic map such that the induced morphism on Lie algebras arises from a homomorphism of $K$-vector spaces $K^n \to \Lie {\sf G}$. Let $\Gamma\subset \C^n$ be a subgroup containing $n$ elements which are $\C$-linearly independent and such that $\varphi(\Gamma)\subset {\sf G}(K)$. Then the algebraic dimension of $\varphi$, i.e. the dimension of the Zariski closure of the image of $\varphi$, is $\leq n$. 
\end{thm}

This implies the following result  \cite[Thm. 3.1]{BC}, which is one of the technical inputs for the  proof of our Theorem \ref{thm:ff1mot} 

\begin{thm}\label{bost-charles}   Let ${\sf G}_K,{\sf H}_K$ be two connected commutative algebraic groups over $K=\bar\Q$. If the group   $H_1({\sf G}_\C)$ generates $\Lie({\sf G}_\C)$ as a complex vector space, then the map
\[\Lie\colon \Hom_{K\text{\rm -gr}}({\sf G}_K,{\sf H}_K)\to \{\psi\in \Hom_{K}(\Lie {\sf G}_K,\Lie {\sf H} _K) | \psi_\C(H_1({\sf G}_\C) )\subseteq H_1({\sf H}_\C)\}
\] 
is an isomorphism of $\Z$-modules.
\end{thm}
 
The condition on the group $H_1({\sf G}_\C)$  is satisfied whenever ${\sf G}_\C$ is a semiabelian variety (or its universal vectorial extension) since $T_\Z({\sf G}_\C)=H_1({\sf G}_\C)$ generates $T_\dr({\sf G}_\C)=\Lie({\sf G}_\C^\natural)$ and hence $\Lie({\sf G}_\C)$. Note that   Theorem \ref{BettideRham} is a generalization of this fact. For completeness we also cite the following result \cite[Thm 1]{WuC}, which is a consequence of the celebrated analytic subgroup theorem of W\"ustholz.  It implies the result of Waldschmidt is a special case. We will use it to give an alternative proof of Theorem \ref{thm:ff1mot}:

\begin{thm}\label{thm:Wusth} Let $W$ be a commutative connected algebraic group over $\bar \Q$. Let $S$ be a subset of $\exp^{-1}\bigl(W(\bar \Q)\bigr)$ and let $V\subset \Lie W$ be the smallest $\bar \Q$-vector subspace whose $\C$-span contains $S$. Then, there exists a connected algebraic subgroup $Z\subset W$ such that $\Lie Z=V$.

\end{thm}

\subsection{Period categories} 
For a fixed $\sigma : K\into \C$ we consider a homological category for Betti-de Rham realizations as follows. Let $\Mod$ be the following category: (i) objects are triples $(H_{\Z}, H_{K}, \omega)$ where $H_\Z$ is a finitely generated abelian group, $H_K$ is a finite dimensional $K$-vector space, and $\omega : H_{\Z}\to H_{K}\otimes_K \C$ is a homomorphism of groups; (ii) morphisms $\phi:(H_{\Z}, H_{K}, \omega)\to (H_{\Z}', H_{K}', \omega')$ are pairs $\phi \df (\phi_\Z , \phi_K)$ where $\phi_\Z \colon H_{\Z}\to H_{\Z}'$ is a group homomorphism, $\phi_K \colon  H_K\to H_K'$ is a $K$-linear homomorphism and $\varphi$ is compatible with $\omega$ and  $\omega'$, \ie  the following square
\begin{equation}\label{4h}
\xymatrix{ H_\Z \ar[r]^(0.4){\omega} \ar[d]_{\phi_\Z}&H_{K}\otimes_K \C\ar[d]^{\phi_K\otimes1_\C}\\
H_\Z' \ar[r]^(0.4){\omega'}&H_{K}'\otimes_K \C
 }\end{equation} commutes. For $H = (H_{\Z}, H_{K}, \omega)$ in $\Mod$ let 
 \begin{equation}\label{omegac}
 \omega_\C : H_{\Z}\otimes_\Z \C\to H_{K}\otimes_K \C
 \end{equation} be the induced $\C$-linear mapping and denote $\Modc$ the full subcategory of $\Mod$ given by those objects such that $\omega_\C$ is a $\C$-isomorphism. 
 
There is a $\Q$-linear variant $\QMod$ of this category where objects are $(H_{\Q}, H_{K}, \omega)$ as above but $H_\Q$ is a finite dimensional $\Q$-vector space. Note that $\QMod \cong \Mod\otimes \Q$ is  the category $\Mod$ modulo torsion objects (see \cite[B.3]{BVK} for this notion).

\begin{defn}
We shall call  $\Modc$ (resp.\/ $\QMod^{\cong}$) the  category (resp.\/ $\Q$-linear category) of {\it homological periods}.
\end{defn}
Let $\Mod^{\cong, \fr}$ (resp. $\Mod^{\cong, \tor}$) be the full subcategory of $\Modc$ given by those objects $H$ such that $H_\Z$ is free (resp. is torsion). For any $r\in \Z$ we shall denote \[\Z (r) \df (\Z, K, (2\pi i)^r)\in \Mod^{\cong, \fr}.\]
For $H = (H_\Z, H_K, \omega)$ and $H' = (H'_\Z, H'_K, \omega')$  we can define 
\begin{equation}\label{twist}
H\otimes H'\df (H_\Z\otimes_\Z H'_\Z,H_K\otimes_K H'_K, \omega\otimes \omega')
\end{equation} and set $H(r)\df H\otimes \Z (r)$ the Tate twist.
For $H\in \Mod^{\cong, \fr}$, say that $H = (H_\Z, H_K, \omega)$ with $H_\Z$ free, we have duals $H^{\vee} \in \Mod^{\cong, \fr}$ given by  
\begin{equation}\label{dual}
(H_\Z, H_K, \omega)^{\vee} \df (H_\Z^{\vee},H_K^{\vee}, \omega^{\vee})
\end{equation}
 where $H_\Z^{\vee} = \Hom (H_\Z, \Z)$ is the dual abelian group, $H_K^{\vee}=\Hom (H_K, K)$ is the dual $K$-vector space, and
$$\omega^{\vee}: H_\Z^{\vee} \to H_K^{\vee}\otimes_K \C $$ is the composition of the canonical mapping $H_\Z^{\vee}\to H_\Z^{\vee}\otimes_\Z \C$ with the $\C$-isomorphism $H_\Z^{\vee}\otimes_\Z \C \longby{\simeq} H_K^{\vee}\otimes_K \C $ given by the inverse of the $\C$-dual of $\omega_\C$ in \eqref{omegac}, i.e., $\omega^{\vee}(f)=(f\otimes_\Z \mathrm{id}_\C)\circ \omega_\C^{-1}$ for any $f\colon H_\Z\to \Z$, up to the canonical isomorphism $H_K^{\vee}\otimes_K \C \simeq (H_K\otimes_K \C)^{\vee} $. We clearly get that $(H^{\vee})^{\vee} = H$ and $(\ \ )^\vee \colon \Mod^{\cong, \fr}\to \Mod^{\cong, \fr}$ is a dualizing functor. Note that $\Z (r)^\vee = \Z (-r)$ so that $H (r)^\vee = H^\vee (-r)$ for $r\in \Z$. 

Similar constructions can be done for the $\Q$-linear variant $\QMod^{\cong}$.  Note that $\QMod^{\cong}$ (resp. $\Mod^{{\cong, \fr}}$) admits an internal $\ihom$ defined via the internal Hom of the category of finite dimensional $\Q$-vector spaces (resp. lattices). Furthermore these categories do have an identity object: $\mathbbm 1=\Z(0)\in \Mod^{{\cong, \fr}}$ and $\mathbbm 1=\Q(0)\in \QMod^{\cong}$, respectively. For any object $H$ of $\QMod^{\cong}$ we have $H^\vee=\ihom(H,\mathbbm 1)$ and $\mathrm{End}(\mathbbm 1)=\Q$. Hence all objects of $\QMod^{\cong}$ are reflexive. Similarly, for $\Mod^{\cong, \fr}$.

\begin{lemma} \label{tensorper}
The categories $\Mod$ and $\Modc$ are abelian tensor categories. The category $\QMod^{\cong}$ is a neutral Tannakian category with fibre functor the forgetful functor to $\Q$-vector spaces.
\end{lemma}

Note that there is a cohomological version of $\Modc$ and $\QMod^{\cong}$, which is called the de Rham--Betti category in the existing literature (\cf \cite[7.5]{An}). 
\begin{defn} Let $\cMod^{\cong}$ be the category whose objects are triples $(H_{K}, H_{\Z}, \eta)$ where $H_{K}$ is a finite dimensional $K$-vector space, $H_{\Z}$ is a finitely generated abelian group and $$\eta : H_{K}\otimes_K \C\longby{\simeq}  H_{\Z}\otimes_\Z \C$$  is an isomorphism  of $\C$-vector spaces. We shall call $\cMod^{\cong}$ and its $\Q$-linear variant $\cQMod^{\cong}$ the categories of {\it cohomological periods}. 
\end{defn}
The category $\cMod^{\cong, \fr} $ is denoted  $\cC_{\drb}$ in \cite[\S 2.1]{BC} and in \cite[\S 5.3]{Bo}.  The $\Q$-linear variant $\cQMod^{\cong}$ is denoted $(K, \Q)$-Vect in \cite[Chap. 5]{Hu}. For these categories we have an analogue of Lemma \ref{tensorper}; in particular, a dualizing functor exists.
\begin{lemma} \label{varsigma}
There is canonical equivalence given by the functor
$$ \varsigma : \Modc\to \cMod^{\cong}\qquad \varsigma (H_\Z, H_K, \omega) \df (H_K, H_\Z, \omega^{-1}_\C)$$
which induces an equivalence between the tensor subcategories  $\Mod^{{\cong, \fr}}$ and $\cMod^{{\cong, \fr}}$. 
\end{lemma}
We set $$\Z (r) \df \varsigma(\Z (r))\in \cMod^{\cong, \fr}.$$
Note that, for $H\in \Mod^{\cong, \fr}$ we may consider $H^\circ\in \cMod^{\cong, \fr}$ setting
\begin{equation}\label{circdual}
 (H_\Z, H_K, \omega)^{\circ} \df (H_K^{\vee}, H_\Z^{\vee}, \omega^{\circ})= \varsigma (H^\vee)= \varsigma (H )^\vee
\end{equation}
where $\omega^{\circ}: H_K^{\vee}\otimes_K \C\longby{\simeq} H_\Z^{\vee}\otimes_\Z \C$ is just given by the $\C$-dual of $\omega_\C$ in \eqref{omegac}. 
We then have $\Z (r)^\circ = \Z (-r)  \in \cMod^{\cong, \fr}$ so that $H(r)^\circ = H^\circ (- r)$ for all $r\in \Z$.

The functor $( \ \ )^{\circ}$ is an anti-equivalence and there is an induced equivalence  $\QMod^{\cong}\cong (\cQMod^{\cong})^{op}$ of neutral Tannakian categories.

\subsection{Betti--de Rham realization and Cartier duality} 
Now recall the period mapping $\varpi_{\sf M,\Z} : T_\Z({\sf M}_\C) \to T_\dr({\sf M}_K)\otimes_K \C$ provided by Definition \ref{perhom}. According to Theorem \ref{BettideRham} we have that $\varpi_{\sf M,\C}$ is a $\C$-isomorphism. 
\begin{defn}\label{Def:bdr}
For $K$ a subfield of $\C$, ${\sf M}_K\in \tM(K)$ and $\varpi_{\sf M,\Z}$ we set $$T_{\bdr}({\sf M}_K)\df (T_\Z({\sf M}_\C), T_\dr({\sf M}_K), \varpi_{\sf M,\Z})\in \Modc$$ 
and the $\Q$-linear variant
$$T_{\bdr}^\Q({\sf M}_K)\df (T_\Q({\sf M}_\C), T_\dr({\sf M}_K), \varpi_{\sf M,\Q})\in \QMod^{\cong}$$  where $T_\Q({\sf M}_\C) \df T_\Z({\sf M}_\C)\otimes_\Z \Q$. Call these realizations the {\it Betti--de Rham} realizations.\end{defn}

Since the period mapping $\varpi_{\sf M,\Z}$ in $T_\bdr ({\sf M}_K)$ is covariantly functorial, by the constructions in \eqref{4sq} and \eqref{iota}, the Betti--de Rham realization yields a functor
\begin{equation}\label{Tfct}
T_{\bdr}  \colon \tM (K) \to  \Modc 
\end{equation} 
in the homological category $\Modc$. Similarly, with rational coefficients, we get a functor  
from $1$-motives up to isogenies $\M^\Q(K) \cong \tM^\Q (K)$ to  $\QMod^{\cong}$. By Examples \ref{e.2} we have $T_{\bdr}(\Z [0]) =\Z (0)$ and $T_{\bdr}(\G_m[-1]) =\Z (1)$.

\begin{defn} For $H=(H_\Z, H_K, \omega)\in \Mod^{{\cong, \fr}}$  define the {\it Cartier dual}\, 
\[H^*:= (H_\Z^\vee, H_K^\vee, 2\pi i\omega^\vee ) = H^\vee(1) = H(-1)^{\vee}=\ihom (H, \Z(1))\in \Mod^{{\cong, \fr}} .\]
\end{defn}

Note that this construction is reflexive.

\begin{thm}\label{Cartier}
For ${\sf M}_K\in \tM (K)$ free with Cartier dual ${\sf M}_K^*$ we have that
$$T_{\bdr} ({\sf M}_K)^* \cong  T_{\bdr} ({\sf M}_K^*)$$
\end{thm}
\begin{proof}It suffices to prove that the Poincar\'e biextension of $ {\sf M}_K$  provides a natural morphism $T({\sf M}_K)\otimes T({\sf M}_K^*)\to\Z(1)$ which induces the usual dualities $\langle\ , \ \rangle_\Z$ on $T_\Z$'s and $\langle\ , \ \rangle_\dr$ on $T_ \dr$'s constructed in \cite[\S 10.2.3 \& \S 10.2.7]{De}.  This is proved in  \cite[Prop. 10.2.8]{De}. 
\end{proof}

Note that we also have a de Rham--Betti contravariant realization in the cohomological category $\cMod^{\cong}$. Recall from \cite[\S 1.13]{BVK} that we also have the category of 1-motives with cotorsion ${}_t\M$. Cartier duality 
\begin{equation}\label{Cdualtor}
(\ \ )^*: {}_t\M \xrightarrow{\simeq} \tM
\end{equation} 
is an anti-equivalence of abelian categories.
\begin{defn}\label{Def:drb}
For  ${\sf M}\in {}_t\M$ denote $$T_\drb ({\sf M})\df \varsigma (T_\bdr ({\sf M}^*)) = (T_\dr ({\sf M}^*), T_\Z ({\sf M}^*), \eta_{{\sf M}^*})\in \cMod^{\cong}$$ where $\eta_{{\sf M}^*}\df \varpi_{{\sf M}^*,\C}^{-1}$ is the inverse of the $\C$-linear period isomorphism $ \varpi_{{\sf M}^*,\C}$ of the Cartier dual ${\sf M}^*\in \tM$ (see Theorem \ref{BettideRham}). Call this realization (and its $\Q$-linear variant) the {\it de Rham--Betti} realization.
\end{defn}
With this definition we get a functor 
\begin{equation}\label{Tfctbis}
T_\drb : {}_t\M^{op} \to\cMod^{\cong}.
\end{equation}
Now we have $T_{\drb}(\Z [0]) =\Z (1)$ and $T_{\drb}(\G_m[-1]) =\Z (0)$.
With the notation adopted in \eqref{circdual},  we also have $$T_\bdr({\sf M})^\circ(1) = (T _\dr ({\sf M})^\vee, T _\Z ({\sf M})^\vee, (2\pi i)^{-1}\varpi_{\sf M}^\circ).$$
\begin{lemma} \label{BdRB}
We have a natural isomorphism of functors $T_\drb (\ \ )\cong T_\bdr(\ \ )^\circ(1)$.\end{lemma}
\begin{proof} For ${\sf M}\in \M$ and its Cartier dual ${\sf M}^*$ we have that $T_\bdr ({\sf M})^* \cong T_\bdr ({\sf M}^*)\in \Mod^{{\cong, \fr}}$ by Theorem \ref{Cartier}. Thus the period isomorphism of the Cartier dual $\varpi_{{\sf M}^*,\Z}= 2\pi i\varpi_{{\sf M},\Z}^\vee$ and its $\C$-inverse $\varpi_{{\sf M}^*,\C}^{-1}= (2\pi i)^{-1}\varpi_{\sf M}^\circ$.
 \end{proof} 

\subsection{Weight and Hodge filtrations}
Consider the category $\fMod$ given by objects in $\Mod$ endowed with  finite and exhaustive filtrations and morphisms that respect the filtrations. 

More precisely, an object of $\fMod$ is an abelian group $H_{\Z}$ endowed with a (weight) filtration $W_{\d}  H_{\Z}$ and a $K$-vector space $H_{K}$ endowed with two filtrations  $W_{\d}  H_{K}$, $F_{\d}  H_{K}$, along with the corresponding compatibilities of the $\omega$'s on weight filtrations.

Let  ${\sf M}_K=[  {\sf L}_K \to {\sf G}_K] $ be a $1$-motive over $K$ and, as usual, let ${\sf T}_K$ denote the maximal subtorus of ${\sf G}_K$. Since the Betti-de Rham realization \eqref{Tfct} is functorial and compatible with the canonical weight filtration  \eqref{eq.weight} on $T_\Z({\sf M}_\C)$    and $T_\dr({\sf M}_{K})$ is filtered by $\V({\sf M})$, the Hodge filtration as in  Remark \ref{hodgefilt}, we also get a realization functor
\begin{equation}\label{FTfct}^{F}T_{\bdr}\colon \tM (K) \to  \fMod.  
\end{equation}
We have $$(T_\Z({\sf M}_\C), T_\dr({\sf M}_K), \varpi_{{\sf M}, \Z})\supseteq (T_\Z({\sf G}_\C), T_\dr({\sf G}_K), \varpi_{{\sf G}, \Z}) \supseteq (T_\Z({\sf T}_\C), T_\dr({\sf T}_K), \varpi_{{\sf T}, \Z}).$$
Note that:

\begin{lemma}\label{weights} Let $K=\bar \Q$. Let ${\sf M}_K$ and ${\sf N}_K$ be two free $1$-motives over $K$. Then any morphism $\varphi\colon T_\bdr({\sf M}_K)\to T_\bdr({\sf N}_K)$ in $\Mod^{{\cong, \fr}}$
preserves the weight filtrations.
\end{lemma}
\begin{proof} Let $\varphi =(\varphi_\Z, \varphi_K) \colon (T_\Z({\sf M}_\C), T_\dr({\sf M}_K), \varpi_{{\sf M}, \Z})\to (T_\Z({\sf N}_\C), T_\dr({\sf N}_K), \varpi_{{\sf N}, \Z})$ for ${\sf M}_K$ and ${\sf N}_K$ one of the following pure $1$-motives: $[\Z_K \to 0]$, $[0\to \G_{m,K}]$ and $[0\to {\sf A}_K]$, where ${\sf A}_K$ is an abelian variety.  We show that $\varphi =0$ for different weights,  in all cases. As $K$ is algebraically closed this implies that $\varphi =0$ for all pure $1$-motives of different weights and this easily yields the claimed compatibility.

For ${\sf M}_K =[\Z_K\to 0]$ and ${\sf N}_K = [0\to \G_{m,K}]$ (respectively   ${\sf M}_K = [0\to \G_{m,K}]$ and ${\sf N}_K =[\Z_K\to 0]$) we have $T_\bdr([\Z_K \to 0]) =\Z (0)$, $T_\bdr([0\to \G_{m,K}]) =\Z(1)$ and $\varphi =0$ as $\varphi_K\colon K\to K$ is given by the multiplication by an algebraic number but the compatibility \eqref{4h} forces such algebraic number to be $n2\pi i$ (respectively $n/2\pi i$) for some $n\in \Z$. 

 Similarly, for ${\sf M}_K =[\Z_K \to 0] $ and ${\sf N}_K =[0\to {\sf A}_{K}]$ we have that $\varpi_{\Z,{\sf N}}\circ \varphi_\Z(1) =\varphi_{K}(1)$ if, and only if, $\varphi=0$. Indeed, the preceding equality implies that $d\rho_{\sf A}\circ \varpi_{\Z,{\sf N}}\circ \varphi_\Z(1)=d\rho_{\sf A}\circ \varphi_{K}(1)$. Now, the right-hand term is in  $\Lie ({\sf A}_K)$ while by Remark \ref{trasc} the left-hand term would give a transcendental point of  $\Lie ({\sf A}_{\C})$   if $\varphi_\Z(1)\neq 0$.

Dually, for ${\sf M}_K =[0\to {\sf A}_K]$ and ${\sf N}_K =[0\to \G_{m,K}]$ by making use of Theorem \ref{Cartier} we then get $\varphi^* =0$ thus $\varphi =0$.

Finally, for ${\sf M}_K = [0\to \G_{m,K}]$ and ${\sf N}_K =[0\to {\sf A}_K]$ we can apply Theorem \ref{bost-charles} to the pair $(\varphi_{K},\varphi_\Z)$  so that, dually, making use of Theorem \ref{Cartier}, the same holds for ${\sf M}_K =[0\to {\sf A}_K]$ and ${\sf N}_K = [\Z_K\to 0]$. 
\end{proof}

Let  ${\sf M}_{K}=[{\sf u}_{K}\colon {\sf L}_K\to {\sf G}_{K}]$ and ${\sf N}_{K}=[{\sf v}_{K}\colon {\sf F}_K\to {\sf H}_K]$ be free and let $\varphi\colon T_\bdr({\sf M}_{K})\to T_\bdr({\sf N}_{K})$ be a morphism in $\Modc$. Then we have a $K$-linear mapping $\varphi_K\colon T_\dr({\sf M}_{K})\to T_\dr({\sf N}_{K})$ and a homomorphism $\varphi_\Z\colon T_\Z({\sf M}_{\C})\to T_\Z({\sf N}_{\C})$ which is compatible with the weight filtrations, by Lemma \ref{weights}. Moreover, $\varphi_\Z$ and $\varphi_K$ are compatible with the  $\varpi$'s as in \eqref{4h}. We have that $\varphi_\Z$ restricts to a homomorphism \begin{equation}\label{w-1}
W_{-1}\varphi_{\Z}\colon W_{-1}T_\Z({\sf M}_\C)\df T_\Z({\sf
G}_\C)\cong H_1({\sf G}_\C^\natural)\to W_{-1}T_\Z({\sf N}_\C)\df T_\Z({\sf
H}_\C)\cong H_1({\sf H}_\C^\natural)
\end{equation}
and we get an induced map on $\gr_0^W$ as follows
 \begin{equation}\label{w-0}
 \varphi_{\Z, 0}\colon \gr_{0}^WT_\Z({\sf M}_\C)=T_\Z({\sf M}_\C)/T_\Z({\sf
G}_\C)= {\sf L}_\C\to \gr_{0}^WT_\Z({\sf N}_\C)=T_\Z({\sf N}_\C)/T_\Z({\sf
H}_\C)= {\sf F}_\C .
\end{equation}
Note that $\varphi_{\Z, 0}$ is indeed defined over $\bar \Q$.

\begin{lemma}\label{hodge} Let $K=\bar \Q$. Let ${\sf M}_K$ and ${\sf N}_K$ be two free $1$-motives over $K$. Then any morphism $\varphi\colon T_\bdr({\sf M}_K)\to T_\bdr({\sf N}_K)$ in $\Mod^{{\cong, \fr}}$ preserves the Hodge filtrations.
\end{lemma}
\begin{proof}
Let  ${\sf M}_{K}=[{\sf u}_{K}\colon {\sf L}_K\to {\sf G}_{K}]$ and ${\sf N}_{K}=[{\sf v}_{K}\colon {\sf F}_K\to {\sf H}_K]$ be free and let $\varphi\colon T_\bdr({\sf M}_{K})\to T_\bdr({\sf N}_{K})$ be a morphism in $\Modc$. We have to show that $\varphi_K(\V ({\sf M}))\subseteq \V ({\sf N})$ where $\V ({\sf M})$ is the additive part of ${\sf G}_K^\natural$ and $\V ({\sf N})$ is that of ${\sf H}_K^\natural$;  see Remark \ref{hodgefilt}. 
Recall the commutative diagram
\begin{equation}\label{4t}
\xymatrix{
  T_\Z({\sf M}_{\C} )\ar[d]^{\varpi_{{\sf M},\Z}}\ar[rr]_{\varphi_{\Z}}  && T_\Z({\sf N}_{\C} )\ar[d]^{\varpi_{{\sf N},\Z}}\\
 T_\dr({\sf M}_{\C} ) \ar[rr]_ {\varphi_{K}\otimes \mathrm{id}_\C}  && T_\dr({\sf N}_{\C})  .
  }
\end{equation} 
By definition of $\varpi_{{\sf M},\Z}$ and \eqref{w-1}  there is then a commutative diagram \begin{equation*} 
\xymatrix{
	 H_1({\sf G}_\C^\natural)\ar[d] \ar[rr]_{W_{-1}\varphi_{\Z}}  &&  H_1({\sf H}_\C^\natural)\ar[d] \\
	\Lie({\sf G}_{\C}^\natural ) \ar[rr]_ {\varphi_{K}\otimes \mathrm{id}_\C}  && \Lie({\sf H}_{\C}^\natural )  
}
\end{equation*}  where the vertical arrows are those in the horizontal sequence in the middle of diagram  \eqref{4sq} for ${\sf M}$ and $\sf N$ respectively. Hence there exists an {\it analytic} morphism   $h_\C \colon {\sf G}_{\C}^\natural\to {\sf H}_{\C}^\natural$ with   $dh_\C= \varphi_K\otimes \mathrm{id}_\C$. 
 It is sufficient to prove that $h_\C$ is algebraic and defined over $K$  to conclude by the structure theorem of  algebraic $K$-groups  that $h_K(\V ({\sf M}))\subseteq \V ({\sf N})$ and hence that $\varphi_K=dh_K$ preserves the Hodge filtrations.

 If ${\sf L}_K=0$, by Lemma \ref{weights}, $\varphi$ factors through $W_{-1}T_\bdr({\sf N}_{K})= T_\bdr({\sf H}_{K})$. Hence we may assume ${\sf F}_K=0$ as well. 
It follows then from Theorem \ref{bost-charles} applied to ${\sf G}^{\natural}$ and ${\sf H}^\natural$ that  the above morphism $h_\C$  is indeed  algebraic and defined over $K$. Hence $\varphi_K$ preserves the Hodge filtrations.
  
 Now let ${\sf L}_K\neq 0$ and set ${\sf L}^\natural_K ={\sf L}_K\otimes \G_{a,K}$. Since $\varphi$ preserves the weights by Lemma \ref{weights} we get $W_{-1}\varphi\colon T_\bdr({\sf G}_K)\to T_\bdr({\sf H}_K)$. By the previous step $W_{-1}\varphi_K (\V ({\sf G}))\subseteq \V ({\sf H})\subseteq \V ({\sf N})$. We thus obtain the following commutative diagram 
 \[
  \xymatrix{T_\Z({\sf M}_\C)\ar[d]_{\varphi_\Z}\ar[r]^{\varpi_{{\sf M}, \Z}\hspace*{1cm}}& T_\dr({\sf M})/\V ({\sf G})\otimes\C  \ar[r]^{\hspace*{0.5cm}\exp}  \ar@{.>}[d]^{\gamma}& {\sf G}^\natural_{\C}/\V ({\sf G})_\C \ar@{=}[r]&{\sf G}_{\C}\oplus {\sf L}_\C^\natural  \ar@{.>}[d]^{\delta}\\
T_\Z({\sf N}_\C)\ar[r]^{\varpi_{{\sf N}, \Z}\hspace*{1cm}}& T_\dr({\sf N})/\V ({\sf N})\otimes\C  \ar[r]^{\hspace*{0.5cm}\exp} & {\sf H}^\natural_{\C}/\V ({\sf N})_\C \ar@{=}[r]&{\sf H}_{\C}  
  }\] 
where the mapping $\gamma$ is induced by $\varphi_K$, we have the canonical identification of ${\sf G}_K^\natural/\V({\sf G}) = {\sf G}_K\oplus {\sf L}_K^\natural$ and $\delta = g_\C + \beta$ with $g_\C = g_K\otimes \mathrm{id}_\C$ and $g_K\colon {\sf G}_K  \to {\sf H}_K$ induced by $W_{-1}\varphi_K$. 
We are left to show that $\beta \colon {\sf L}_\C^\natural  \to {\sf H}_{\C}$ is zero. Since the composition of the upper arrows in the previous diagram maps $T_\Z({\sf G}_\C) $ to $0\oplus 0$, we obtain a commutative square
\[
  \xymatrix{{\sf L}_\C\ar[d]_{\varphi_{\Z , 0}}\ar[r]^{({\sf u}, 1)\ \ \ }  & {\sf G}_{\C}\oplus {\sf L}_\C^\natural  \ar[d]^{\delta}\\
{\sf F}_\C\ar[r]^{{\sf v}}& {\sf H}_{\C}  
  }\] 
where $\varphi_{\Z , 0}$ is the induced map as in \eqref{w-0}. In particular, for $x\in {\sf L}_K(K)$ we have
$\beta (x\otimes 1)= {\sf v} (\varphi_{\Z , 0}(x))- g_K({\sf u}(x))= \gamma -dg_K\otimes \mathrm{id}_\C$ is in ${\sf H}_K(K)$. On the other hand $ \beta (x\otimes 1) = \exp d\beta (x\otimes 1)$. Since $ d\beta = d\delta - d g_K\otimes\mathrm{id}_\C$ we have that $d\beta (x\otimes 1)$ belongs to $\Lie({\sf H}_K)$ regarded as a $K$-linear subspace of $\Lie ({\sf H}_\C)$. By Remark \ref{trasc} we get that $\beta (x\otimes 1)=0$ and therefore that $ \beta =0$. 
\end{proof}

\subsection{Full faithfulness}
We are now ready to show that our previous Lemmas \ref{weights} and \ref{hodge} yield the full faithfulness of Betti--de Rham and de Rham--Betti realizations.
\begin{thm} \label{thm:ff1mot}
The functors $T_\bdr$ in \eqref{Tfct} and $T_\drb$ in \eqref{Tfctbis} restricted to $\M (K)$ are fully faithful over $K = \bar \Q$.
\end{thm}
\begin{proof} Clearly, the functor $T_\bdr$ (resp. $T_\drb$) is faithful (\cf  \cite[proof of Lemma 3.3.2]{ABVB}) and we are left to show the fullness. Making use of Lemma \ref{BdRB} we are left to check the fullness for $T_\bdr$. Let  ${\sf M}_{K}=[{\sf u}_{K}\colon {\sf L}_K\to {\sf G}_{K}]$ and ${\sf N}_{K}=[{\sf v}_{K}\colon {\sf F}_K\to {\sf H}_K]$ be free and let $\varphi\colon T_\bdr({\sf M}_{K})\to T_\bdr({\sf N}_{K})$ be a morphism in $\Modc$. 

For $0$-motives, \ie if ${\sf G}_{K} = {\sf H}_{K} =0$, we have ${\sf L}_K\cong \Z^r_{K}$ and ${\sf F}_K \cong \Z^s_{K}$,  $\varphi_\Z \colon T_\Z({\sf M}_{\C}) \cong \Z^r \to T_\Z({\sf N}_{\C})\cong \Z^s_{K}$ provides a morphism $f: {\sf M}_{K}\to {\sf N}_{K}$ such that $T_\bdr(f) = \varphi$.

If the weight $-1$ parts are non-zero, by Lemma \ref{weights} $\varphi_\Z$ restricts to a homomorphism $W_{-1}\varphi_\Z$ as in \eqref{w-1} and it yields a morphism $\varphi_{\Z, 0}$ as in \eqref{w-0} \ie $\varphi_{\Z, 0}$ is the map induced  by $\varphi_\Z$ on $\gr_0^W$. If we set $f_\C\df\varphi_{\Z, 0}$ the homomorphism $f_\C\colon {\sf L}_\C\to  {\sf F}_\C$ trivially descends to a homomorphism $f_K \colon {\sf L}_{K}\to  {\sf F}_{K}$ over $K=\bar \Q$. 
 
Let's now consider $\varphi_\C\df  \varphi_K\otimes_K\mathrm{id}_\C$ and translate \eqref{w-1} and  \eqref{4t}, as in the proof of Lemma \ref{hodge},  in the following commutative diagram with exact rows
  \begin{equation}\label{dia1}
  \xymatrix{0\ar[r]& H_1({\sf G}_\C^\natural )  \ar[r]  \ar[d]_{W_{-1}\varphi_\Z}& \Lie({\sf G}_{\C}^\natural)\ar[r]^(.6){\exp}\ar[d]^{\varphi_\C}&{\sf G}_{\C}^\natural\ar[r]\ar@{.>}[d]^{\psi}&0\\
  0\ar[r]& H_1({\sf H}_\C^\natural )  \ar[r]  & \Lie({\sf H}_{\C}^\natural)\ar[r]^(.6){\exp}&{\sf H}_{\C}^\natural\ar[r]&0
  }\end{equation}
yielding a morphism of analytic groups $\psi \colon {\sf G}_\C^\natural  \to {\sf H}_\C^\natural$ on the quotients via the exponential mapping $\exp$, as indicated above. 
Now, since by Lemma \ref{hodge}, we have $\varphi_K (\V({\sf M}))\subseteq \V({\sf N})$, $\psi(\V({\sf M_{\C}}))\subseteq \V({\sf N}_\C)$, the diagram \eqref{dia1} induces a commutative diagram 
\[
  \xymatrix{0\ar[r]& H_1({\sf G}_\C  )  \ar[r]  \ar[d]_{W_{-1}\varphi_\Z}& \Lie({\sf G}_{\C} )\ar[r]^{\exp}\ar@{.>}[d]^{\varphi'_\C}&{\sf G}_{\C} \ar[r]\ar@{.>}[d]^{\psi'}&0~\\
  0\ar[r]& H_1({\sf H}_\C  )  \ar[r]  & \Lie({\sf H}_{\C} )\ar[r]^{\exp}&{\sf H}_{\C} \ar[r]&0.
  } \] 
As $\varphi'_\C$ is the base change of the $K$-linear map $\Lie({\sf G}_{K}) \to \Lie({\sf H}_{K})$  induced by $\varphi_K$, it follows from Theorem \ref{bost-charles} that  $\psi'= g_\C$ is the base change of the morphism $g_{K}\colon {\sf G}_K\to {\sf H}_{K}$ over $K=\bar \Q$ induced by $W_{-1}\varphi_K$ (see the proof of Lemma \ref{hodge}).

We are left to check that $h\df (f_{K},g_{K})$ gives a morphism $h\colon {\sf M_{K}}\to {\sf N}_{K}$, i.e., that $g_{K}\circ {\sf u}_{K}={\sf v}_{K}\circ f_{K}$, and to see that $T_\bdr(h) = \varphi$. To show that $h$ is a morphism of $1$-motives we may work after base change to $\C$ and, using \eqref{4sq}, it suffices to prove that $\psi\circ {\sf u}_{\C}^\natural={\sf v}_{\C}^\natural\circ f_{\C}$. Consider the following diagram
 \[\xymatrix{ T_\Z({\sf
M}_\C) \ar[rr]^{ \varphi_{\Z}} \ar[dd]_(.3){\varpi_{\sf M,\Z}}\ar[dr]^{\widetilde \exp}& &T_\Z({\sf
N}_\C)\ar[dd]^(.3){\varpi_{\sf N,\Z}}\ar[dr]^{\widetilde \exp}&\\
&{\sf L}_\C \ar[dd]_(.3){{\sf u}_\C^\natural}\ar[rr]^(.3){ f_{\C}} && {\sf F}_\C \ar[dd]^(.3){{\sf v}_\C^\natural}~~\\
T_\dr({\sf M}_K)\otimes_K\C\ar[rr]_(.4){\varphi_\C}\ar[dr]_{\exp}&&T_\dr({\sf N}_K)\otimes_K\C
 \ar[dr]_{\exp}& \\
 &{\sf G}_\C^\natural\ar[rr]^(.3){\psi}  && {\sf H}_\C^\natural ~.} \] 
 All squares are commutative. Indeed, $\exp\circ \varpi_{\sf M,\Z}= {\sf u}_\C^\natural\circ \widetilde \exp$ and $\exp\circ \varpi_{\sf N,\Z}= {\sf v}_\C^\natural\circ \widetilde \exp$  by \eqref{4sq}, $f_\C\circ \widetilde \exp=  \widetilde \exp \circ \varphi_\Z$ by definition of $f_\C$, $\varphi_\C\circ \varpi_{\sf M,\Z}= \varpi_{\sf N,\Z} \circ\varphi_\Z$ by  the compatibility of $\varphi_\Z$ with $\varphi_K$ as in \eqref{4h}, and finally $\psi\circ \exp=\exp\circ \varphi_\C$ by \eqref{dia1}. One concludes by the surjectivity of the map  $\widetilde \exp$ that also $\psi\circ {\sf u}_{\C}^\natural={\sf v}_{\C}^\natural\circ f_{\C}$. 

Now consider the   morphism $\alpha\df T_\bdr(h) -\varphi\colon T_\bdr({\sf M}_{K})\to T_\bdr({\sf N}_{K})$ in $\Modc$. By construction $W_{-1}\varphi_\Z = W_{-1}T_\Z(h)$ so that $\alpha$ is vanishing on $T_\bdr({\sf G}_{K})$. Moreover we have that $\gr_0^WT_\Z(h) = \varphi_{\Z, 0}$ so that $\alpha$ induces a morphism in $\Modc$ from $T_\bdr({\sf L}_{K})$ to $T_\bdr({\sf H}_{K})$ which is trivial by Lemma \ref{weights}. 
\end{proof}

\begin{remark}
In the proof of Theorem \ref{thm:ff1mot}, in order to show that $(f_K,g_K)\colon {\sf M}_{K}\to {\sf N}_{K}$ is a morphism we are left to check that $g_K \circ {\sf u}_{K}={\sf v}_{K} \circ f_K$. Remark that this also follows from two key facts: (i) the pullbacks $\tilde{\sf u}_{\C}$ and $\tilde{\sf v}_{\C}$ of ${\sf u}_{\C}$ and ${\sf v}_{\C}$ factor through the period mappings $\varpi$'s and (ii) the mappings $\varphi_\Z$ and $\varphi_K$ are compatible with the $\varpi$'s. 

In fact, according to the above notation, for $x\in {\sf L}_\C$ pick $\tilde{\log} (x)\in T_\Z({\sf M}_\C)$ and note that 
$\varphi_\Z(\tilde{\log}(x))= \tilde{\log}(f(x))$. Making use of \eqref{logper} we obtain
$$ g_\C ({\sf u}_\C (x))= g_\C\exp d\rho_{\sf G}\varpi_{M,\Z}(\tilde{\log}(x))= \exp d\rho_H dg^\natural_\C\varpi_{M,\Z}(\tilde{\log}(x))$$ by the functoriality of $\exp$. Now $\psi =\varphi_K\otimes_K1_\C$ and we are assuming the compatibility $(\varphi_K\otimes_K1_\C)\circ\varpi_{M,\Z}=\varpi_{N,\Z}\circ\varphi_\Z$ so that
$$ g_\C ({\sf u}_\C (x)) = \exp d\rho_H \varpi_{N,\Z}\varphi_\Z (\tilde{\log}(x))= \exp d\rho_H \varpi_{N,\Z}(\tilde{\log}(f(x))) = {\sf v}_{\C} (f(x))$$
using \eqref{logper} again, as claimed.
\end{remark}

We notice that an alternative proof of Theorem \ref{thm:ff1mot} can be given using W\"ustholz's analytic subgroup Theorem  \ref{thm:Wusth} as follows:\\

{\it Alternative Proof of Theorem \ref{thm:ff1mot}.}
Let  ${\sf M}_{K}=[{\sf u}_{K}\colon {\sf L}_K\to {\sf G}_{K}]$ and ${\sf N}_{K}=[{\sf v}_{K}\colon {\sf F}_K\to {\sf H}_K]$ be free $1$-motives and let 
$\varphi=(\varphi_\Z,\varphi_K) \colon T_\bdr({\sf M}_{K})\to T_\bdr({\sf N}_{K})$ be a morphism in $\Modc$.  Let $ W={\sf G}_K^\natural \times {\sf H}_K^\natural$ and note that  we have   commutative squares
\[\xymatrix{
T:=T_\Z({ \sf  M}_\C)\ar[rr]^{(id,\varphi_\Z)}\ar[d] &&T_\Z({ \sf  M}_\C)\times T_\Z({ \sf  N}_\C)\ar[d]\ar[rr]^(.45){\varpi_{{\sf M},\Z}\times  \varpi_{{\sf N},\Z}}&&   \Lie W_{\C}=\Lie{\sf G}_{\C}^\natural  \oplus \Lie{\sf H}_{\C}^\natural \ar[d]^{exp}\\
{ \sf  L}_\C\ar[rr]^{(id,\varphi_{\Z,0})} &&{ \sf  L}_\C\times { \sf  F}_\C \ar[rr]^(.4){ {\sf u}^\natural\times {\sf v}^\natural}&& W_\C:={\sf G}_{\C}^\natural \times { \sf H}_{\C}^\natural
}
\]
where the horizontal arrows are injective.
Let $S$ denote the image of $T$ in $ \Lie W_{\C} $; it  is contained in $ \exp^{-1}\bigl(W(\bar \Q)\bigr)$ since the image of  ${\sf L}_{\C} \times  {\sf F}_{\C} $ via $ {\sf u}^\natural\times {\sf v}^\natural $   is contained in $W(\bar \Q)$.  
Let $V$ denote the image of $\Lie {\sf G}_K^\natural$  in $\Lie W$ via the map ${\rm id}\oplus \varphi_K$.   By the compatibility of $\varphi_\Z$ and $\varphi_K$ over $\C$ via the homomorphisms of periods,  $V_\C$ coincides with the $\C$-span of $S$.
It then follows from Theorem \ref{thm:Wusth}  that there exists an algebraic subgroup $Z\subset W$ whose Lie algebra is $V$. 
 Now, the composition of the inclusion $Z\to  W$ with the projection $W\to {\sf G}_K^\natural$ is an isogeny, since it is an isomorphism on Lie algebras. In fact,  it is an isomorphism; indeed the injective map $T\to V_\C=\Lie Z_\C\subset \Lie W_\C$ maps $H_1({\sf G}^\natural_\C)\subset T$ into $H_1(Z_\C)\subset H_1(W_\C)$ and hence the isomorphism $\Lie Z_\C\longby{\sim} \Lie {\sf G}^\natural_\C$ restricts to an isomorphism 
 $H_1(Z_\C)\longby{\sim} H_1({\sf G}^\natural_\C )$.
 Let $\gamma \colon  {\sf G}_K^\natural \to {\sf H}_K^\natural$ be the homomorphism of algebraic $K$-groups defined by composing the inverse of the isomorphism 
$Z\to {\sf G}_K^\natural$ with the inclusion $Z\to W$ and the second projection $W\to  {\sf H}_K^\natural$. By construction $\Lie \gamma=\varphi_K$.

In order to see that $f:=(\varphi_{\Z,0},\gamma)$ is a morphism of $1$-motives with $T_\bdr(f)=\varphi$ it suffices to check that $\gamma_\C\circ {\sf u}^\natural_\C={\sf v}^\natural_\C\circ \varphi_{\Z,0}$ as morphisms $ {\sf L}_\C\to H_\C^\natural$. The latter fact is equivalent to the equality
$({\rm id}_{{\sf G}^\natural},\gamma)\circ {\sf u}^\natural=({\sf u}^\natural,{\sf v}^\natural\circ \varphi_{\Z,0})$ as morphisms ${\sf L}_\C\to W_\C$, and, by the above diagram, this is satisfied whenever 
$({\rm id}_{\Lie{\sf G}^\natural},\Lie \gamma_\C)\circ \varpi_{{\sf M},\Z}=(\varpi_{{\sf M},\Z},\varpi_{{\sf N},\Z}\circ\varphi_\Z)\colon T\to \Lie W_\C$.
Then we conclude by the commutativity of diagram \eqref{4h} since $\Lie \gamma_\C=\varphi_K\otimes {\rm id}_\C$.
\eprooff

\subsection{Descent to number fields}
Let $K'/ K$ be a field extension with $K'\subseteq \bar \Q$. Note the following commutative diagram of functors
\begin{equation}\label{4m}
\xymatrix{ \M(K) \ar[r]^{T_\bdr} \ar[d] & \Mod^{{\cong, \fr}} \ar[d]  \\
 \M(K') \ar[r]^{T_\bdr}  &\mathrm{Mod}_{\Z,K'}^{{\cong, \fr}}
 }\end{equation} 
where the functor on the left is the usual base-change and the vertical functor on the right  maps $(H_\Z,H_K,\omega)$ to $(H_\Z,H_K\otimes_K K',\omega)$ using the canonical isomorphism $(H_K\otimes_K K')\otimes_{K'} \C\simeq  H_K\otimes_K \C$.

\begin{propose}
Let $K$ be a subfield of $\bar\Q$. The functor $T_\bdr\colon \M(K) \to \Mod^{{\cong, \fr}} $ is fully faithful.
\end{propose}
\begin{proof}
The functor  $T_\bdr$ is fully faithful over $\bar\Q$ by Theorem \ref{thm:ff1mot};  hence it is faithful over $K$, since the left-hand vertical functor in \eqref{4m} is faithful.  

Assume now ${\sf M}_K=[{\sf u}\colon {\sf L}_K\to {\sf G}_K], {\sf N}_K=[{\sf v}\colon {\sf F}_K\to {\sf H}_K]$ are $1$-motives over $K$ and let $(\varphi_\Z,\varphi_K)\colon T_\bdr({\sf M}_K)\to T_\bdr({\sf N}_K)$ be a morphism in $\Mod$. By Theorem \ref{thm:ff1mot} there exists a morphism $\psi\colon {\sf M}_{\bar\Q}\to {\sf N}_{\bar\Q}$ such that $T_\bdr(\psi)=(\varphi_\Z,\varphi_K\otimes_K{\rm id}_{\bar\Q})$. 
Note that there exists a subfield  ${K'}\subset \bar \Q$ with ${K'}/K$ finite Galois and $\psi=(f,g)$ is defined over ${K'}$. 
We may further assume that ${\sf L}_{K'}, {\sf F}_{K'}$ are constant free.
Hence $\gr_0^W\varphi_\Z$ descends over $K'$ and we have a commutative square
\begin{equation}\label{lfgg} \xymatrix{ {\sf L}_{K'} \ar[rr]^{f=\gr_0^W(\varphi_\Z)} \ar[d]  && {\sf F}_{K'}  \ar[d]   \\
   {\sf L}_{K'}\otimes \G_{a,K'} \ar[rr]^{\gr_0^W(\varphi_{K'})}  && {\sf F}_{K'} \otimes \G_{a,K'}  
}\end{equation}
where the vertical morphisms map  $x$ to $x\otimes 1$ (and descend the homomorphism of periods for $\gr_0^W({\sf M})$ and $\gr_0^W({\sf N})$ respectively).
By diagram \eqref{lfgg} $f$ descends over $K$ since $\gr_0^W(\varphi_{K'})=\gr_0^W(\varphi_K) \otimes \mathrm{id}_{K'}$ and the vertical morphisms are injective on points. In order to check that $\psi$ descends over $K$, we may then reduce to the case ${\sf L}_K={\sf F}_K=0$. By Cartier duality, we may further reduce to the case where ${\sf L}_K={\sf F}_K=0$ and ${\sf G}_K=\sf{A}_K, {\sf H}_K={\sf B}_K$ are abelian varieties.

For any $\tau\in \mathrm{Gal}({K'}/K)$ let $\tau$ also denote  the corresponding $K$-automorphism of $\Spec K'$. Further let $\tau^*{\sf A}_{K'}$ denote the base change of ${\sf A}_{K'}$ along $\tau$ and let $\tau_{{\sf A}_{K'}}\colon  {\sf A}_{K'}\to {\sf A}_{K'}$ be equal to ${\rm id}_{\sf A_K}\otimes \tau $. It is not a morphism of $K'$-schemes in general. Finally let $ \iota_{{\sf A},\tau}\colon {\sf A}_{K'}\to \tau^*{\sf A}_{K'}$ be the canonical morphism of $K'$-schemes such that $\tau_{\sf A_{K'}}$ is the composition of $ \iota_{{\sf A},\tau}$ with the projection $\tau^*{\sf A}_{K'}\to{\sf A}_{K'} $.  

In order to prove that the morphism   $\psi\colon {\sf A}_{K'}\to {\sf B}_{K'}$ is defined over $K$ we have to check that for any $\tau\in \mathrm{Gal}({K'}/K)$ it is $\tau_{{\sf B}_{K'}}\circ\psi=\psi\circ \tau_{{\sf A}_{K'}}$. 
In fact it is sufficient to check that $\iota_{{\sf B},\tau}\circ \psi=(\tau^*\psi)\circ \iota_{{\sf A},\tau}$ as morphisms of $K'$-schemes ${\sf A}_{K'}\to \tau^*{\sf B}_{K'}$, where $\tau^*\psi\colon \tau^*{\sf A}_{K'}\to\tau^*{\sf B}_{K'}$ is the obvious base change of $\psi$.
By faithfulness of $T_\bdr$, it is sufficient to check that  
\begin{equation}
\label{4tbdr}
T_\bdr(\iota_{{\sf B},\tau})\circ T_\bdr(\psi)=T_\bdr(\tau^*\psi)\circ T_\bdr( \iota_{{\sf A},\tau}).\end{equation}
Note that since ${\sf A}_K$ is a $K$-form of ${\sf A}_{K'}$, we may identify ${\sf A}_{K'}$ with $\tau^* {\sf A}_{K'}$ so that $\iota_{{\sf A},\tau}$ becomes the identity map. Further $T_\dr(\tau^*\psi)=\tau^*(\varphi_K\otimes \mathrm{id}_{K'})$ may be identified with $\varphi_K\otimes \mathrm{id}_{K'}$ and $T_\Z(\tau^*\psi)$ with $\varphi_\Z$. We conclude that $ T_\bdr(\tau^*\psi)$ may be identified with $(\varphi_\Z,\varphi_K\otimes \mathrm{id}_{K'})$ and hence  \eqref{4tbdr} is clear.
\end{proof}

\section{Some evidence: description of some Grothendieck arithmetic invariants}
Throughout this section we assume that $K=\bar \Q$ and by scheme we mean a separated scheme  of finite type over $K$. 
In order to show the period conjecture for motivic cohomology \eqref{periodconj} we are left to deal with rational coefficients. However, we prefer to keep the arguments integral when possible.
In general, for any algebraic scheme $X$ over $K$, by making use of the period isomorphism $\varpi^{p,q}_X$  and its inverse $\eta^{p,q}_X$ in Definition \ref{classper} we set
$$H_{\bdr}^{p, q}(X) \df (H^p(X_\an, \Z_\an (q)), H^p_\dr(X), \varpi^{p,q}_X)\in \Modc$$
and 
$$H_{\drb}^{p, q}(X) \df (H^p_\dr(X), H^p(X_\an, \Z_\an (q)), \eta^{p,q}_X)\in \cMod^{\cong}.$$
Note that $\varsigma (H_{\bdr}^{p, q}(X)) = H_{\drb}^{p, q}(X)$. 
We have that $\varpi^{p,q}_X= (2\pi i)^{q}\varpi^{p,0}_X$ and $\eta^{p,q}_X= (2\pi i)^{-q}\eta^{p,0}_X$ where $\eta^{p,0}_X: H^p_\dr(X)\otimes_K\C \longby{\simeq}  H^p(X_\an, \C)$ is the usual de Rham--Betti comparison isomorphism (up to a sign \cf \cite[Def. 5.3.1]{Hu} and \cite[Lemma 4.1.1 \& Prop. 4.1.2]{LW} for the Nisnevich topology). 
In particular we have that $H_{\drb}^{p, q}(X)= H_{\drb}^{p, 0}(X)(q)$.
\subsection{Period cohomology revisited}
For $H\in \Mod$ we set $$H_{\varpi} \df \Hom (\Z(0), H)$$ where the $\Hom$-group is taken in $\Mod$. This yields a functor $$( \ \  )_\varpi :  \Mod\to \Ab$$ to the category of finitely generated abelian groups. 
Similarly, let  $H_{\varpi} \df \Hom (\Z(0), H)$ for $H \in \cMod^{\cong}$ where now the $\Hom$-group is taken in $\cMod^{\cong}$. By Lemma \ref{varsigma}, for $H \in \Modc$ we clearly have that $$H_{\varpi}= \varsigma(H)_\varpi.$$ 
Moreover, for $H\in \Mod^{{\cong, \fr}}$ we have $H^* = H(-1)^\vee \in \Mod^{{\cong, \fr}}$ so that  $$H_{\varpi}^* =\Hom  (\Z(0), H(-1)^\vee) = \Hom  (H , \Z(1)).$$
Note that for $H\in \Mod^{{\cong, \fr}}$ we also have $H^\circ(1)= H (-1)^\circ\in \cMod^{{\cong, \fr}}$ (see \eqref{circdual}) and we shall denote  $$H^{\varpi} \df\Hom_{\cMod^{{\cong, \fr}}} (\Z(0), H (-1)^\circ)= \Hom_{\Mod^{{\cong, \fr}}} (H, \Z (1)) = H_{\varpi}^*.$$
With rational coefficients, for $H\in \Mod\otimes \Q$ and the corresponding $H^\Q\in \QMod$ we then have $H_{\varpi}\otimes_\Z \Q \cong H_{\varpi}^\Q\df \Hom (\Q(0), H^\Q)$ and, similarly, $H^{\varpi}\otimes_\Z \Q\cong H^{\varpi}_\Q$. 
We have (\cf\cite[Def. 2.1]{BC} and \cite[(5.15)]{Bo}):
\begin{lemma}\label{Homega}
For $H=(H_\Z, H_K, \omega)\in \Mod$ we have
$H_{\varpi} \cong H_\Z\cap H_K$
where $\cap$ is the inverse image of $H_K$ under $\omega :H_\Z\to H_K\otimes_K\C$. Moreover, for $H=(H_K, H_\Z,\eta)\in \cMod^{\cong}$, we have that $H_{\varpi}\cong H_K\cap H_\Z$ where $\cap$ is the inverse image of $H_K$ under the composition of $H_\Z\to H_\Z\otimes_\Z\C\longby{\eta^{-1}} H_K\otimes_K\C$.
\end{lemma}
\begin{proof}
The identifications are given by mapping $\varphi\in \Hom (\Z(0), H)$ to $\varphi (1)\in H_\Z\cap H_K$. 
\end{proof}
Similarly, for $H = (H_K, H_\Q, \omega)\in \cQMod^{\cong}$ we have that $H_{\varpi}^\Q \cong H_K\cap H_\Q$. We then clearly obtain:
\begin{cor}\label{perhombis}
For $H= H_{\drb}^{p, q}(X)$ we have that $H_{\varpi} \cong H^{p,q}_\varpi (X)$
coincides with the period cohomology of Definition \ref{defper}. 
\end{cor}
Moreover, composing the functor $H\leadsto H_\varpi$ with the Betti--de Rham realization of 1-motives $T_{\bdr}$ in \eqref{Tfct} we obtain a functor
\begin{equation} \label{Tvarpi}
T_{\varpi} \df  (\ \ )_\varpi  \circ T_\bdr: \tM (K) \to \Ab.
\end{equation}
For a $1$-motive ${\sf M} \in {}_t\M (K)$ we also have $T_\drb ({\sf M})\in \cMod^{\cong}$.
Composing $H\leadsto H_\varpi$ with the de Rham--Betti realization $T_\drb$ in \eqref{Tfctbis} now yields a functor  
$$T^\varpi \df (\ \ )_\varpi  \circ T_\drb: {}_t\M (K)^{op} \to \Ab.$$ 
We also note that Lemma \ref{BdRB} yields:
\begin{cor} For ${\sf M}\in \M (K)$ we have that 
$T^\varpi ({\sf M})\df T_\drb({\sf M})_\varpi\cong T_\bdr ({\sf M})^\varpi$.
\end{cor}
Working with rational coefficients we have $T_\bdr^\Q\df T_\bdr\otimes\Q$ (resp.\/ $T_\drb^\Q\df T_\drb\otimes\Q$)  and we then get a functor $T_\varpi^\Q$ (resp.\/ a contravariant functor $T^\varpi_\Q$) from the category of 1-motives up to isogenies $\M^{\Q}\df \M\otimes \Q\cong \tM\otimes \Q\cong {}_t\M\otimes \Q$ to the category of finite dimensional $\Q$-vector spaces.
Moreover, applying our Theorem \ref{thm:ff1mot} we have:
\begin{cor} \label{keru}
For ${\sf M}=[{\sf u}\colon {\sf L}\to {\sf G}] \in \M (\bar \Q)$ with Cartier dual   ${\sf M}^*=[{\sf u}^*\colon {\sf L}^*\to {\sf G}^*] \in \M (\bar \Q)$ we have that
$$T_{\varpi} ({\sf M})\cong T_{\Z}({\sf M}_\C) \cap T_\dr ({\sf M}_K) \cong \ker {\sf u}  $$
and 
$$T^{\varpi} ({\sf M})\cong  T_\dr ({\sf M}^*_K)\cap T_{\Z}({\sf M^*_\C}) \cong \ker {\sf u}^*. $$
\end{cor}
\begin{proof} Note that $\Z(0) = T_\bdr (\Z [0])$ and for $T_\bdr ({\sf M})= (T_\Z({\sf M}_\C), T_\dr({\sf M}_K), \varpi_{\sf M,\Z})$ we have 
$$\ker {\sf u}\cong \Hom_{\M (K)} (\Z [0],  {\sf M}) \longby{T_\bdr}\Hom (T_\bdr (\Z [0]),  T_\bdr ({\sf M})) = T_{\varpi}({\sf M})$$
which is an isomorphism over $K =\bar \Q$ as proven in Theorem \ref{thm:ff1mot}. We just apply Lemma \ref{Homega}.  Moreover, $T_\drb ({\sf M})= (T_\dr({\sf M}^*_K), T_\Z({\sf M}^*_\C), \eta_{\sf M^*})$, $\Z(0) = T_\drb(\G_m [-1])$ and we have an isomorphism
$$\Hom_{\M (K)} ({\sf M}, \G_m [-1]) \longby{T_\drb}\Hom (T_\drb(\G_m [-1]),  T_\drb ({\sf M})) = T^{\varpi}({\sf M})$$
and $\Hom_{\M (K)} ({\sf M}, \G_m [-1])\cong \Hom_{\M (K)} (\Z[0], {\sf M}^*) = \ker {\sf u}^*$ showing the claim.\end{proof}

\subsection{Period conjecture for $q=1$}
Recall that $\Z (1)\in \DM^\eff_{\et}$ is canonically identified with $\Tot ([0\to \G_m])=\G_m [-1]$ (see \cite[Lemma 1.8.7]{BVK}). We then have 
$$H^{p, 1}(X)\cong H^{p-1}_\eh(X,\G_m)$$ for all $p\in \Z$.
Recall the motivic Albanese triangulated functor
$$\LAlb  : \DM^\eff_{\gm}\to D^b (\M )$$ 
where $\DM^\eff_{\gm}\subset \DM^\eff_{\Nis}$ is the subcategory of compact objects, \ie the category of geometric motives, which has been constructed in \cite[Def. 5.2.1]{BVK} (see also \cite[Thm. 2.4.1]{BVA}). This functor is integrally defined. Rationally, $\LAlb$ yields a left adjoint to the inclusion functor given by $\Tot$ in \eqref{tot} (see \cite[Thm. 6.2.1]{BVK}).

Applying $\LAlb$ to the motive of any algebraic scheme $X$ we get $\LAlb(X) \in D^b(\M )$, a complex of 1-motives whose $p$-th homology $\LA{p} (X) \in{}_t\M$ is a 1-motive (with cotorsion, see \cite[Def. 8.2.1]{BVK}). Dually, we have $\RPic (X)\in D^b(\tM)$ (see \cite[\S 8.3]{BVK}). Taking the Cartier dual of $\LA{p} (X)$ we get $\RA{p}(X)\in \tM$ and conversely via \eqref{Cdualtor}. Now, the motivic Albanese map $$M(X)\to \Tot \LAlb (X)$$ in $\DM^\eff_{\et}$ (see \cite[\S 8.2.7]{BVK}) yields an integrally defined map
\begin{equation}\label{albmap}
\Hom_{D^b({}_t\M)} (\LAlb (X), [0\to \G_m][p])\to \Hom_{\DM^\eff_{\et}} (M(X), \Z(1)[p])\cong H^{p-1}_\eh(X,\G_m).
\end{equation}
Rationally (by adjunction), this map becomes a $\Q$-linear isomorphism 
\begin{equation}\label{adjoint}
H^{p-1}_\eh(X,\G_m)_\Q \cong \Hom_{\DM^\eff_{\et, \Q}} (M(X), \Z(1)[p])\xleftarrow{\simeq}\Hom_{D^b(\M^\Q)} (\LAlb (X), \G_m[-1][p]).
\end{equation}
Using \eqref{Cdualtor} we set $$\EExt^p (\Z, \RPic (X))\df \Hom_{D^b(\tM)} (\Z, \RPic (X)[p])\cong \Hom_{D^b({}_t\M)} (\LAlb (X), \G_m[-1][p])$$ for all $p\in \Z$ and we also have (\cf \cite[Lemma 10.5.1]{BVK}):
\begin{lemma}\label{reproj} For any $X$ over $K =\bar \Q$ and $p\in \Z$ there is an extension 
$$0\to \Ext (\Z, \RA{p-1}(X))\to \EExt^p (\Z, \RPic (X))\longby{\pi}\Hom (\Z, \RA{p}(X))\to 0$$
where the $\Hom$ and $\Ext$ are here taken in the category $\tM$ of $1$-motives with torsion. The composition of \eqref{albmap} with the period regulator $r_\varpi^{p, 1}:H^{p-1}_\eh(X,\G_m)\to H^{p,1}_\varpi (X)$ induces a mapping 
$$\theta_\varpi^{p}: \Hom (\Z, \RA{p}(X))\to H^{p,1}_\varpi (X).$$
\end{lemma}
\begin{proof} In fact,  the canonical spectral sequence
$$E^{p,q}_2 = \Ext^p(\Z, \RA{q}(X))\ \ \implies\ \ \EExt^{p+q} (\Z, \RPic (X))$$
yields the claimed extension since the abelian category of 1-motives with torsion $\tM (K)$ is of homological dimension $1$ over the algebraically closed field $K = \bar \Q$.
Moreover, for any 1-motive ${\sf M}= \RA{p}(X)\in \tM (K)$ the group $\Ext (\Z, {\sf M})$ is divisible and the group $\Hom (\Z, {\sf M})$ is finitely generated (as it follows easily by making use of \cite[\S C.8]{BVK}). The horizontal mapping in the following commutative diagram
\[
\xymatrix{
\Ext (\Z, \RA{p-1} (X))\ar@{_{(}->}[d]_{ }\ar@/^1.6pc/[drr]^-{\text{zero}}&&\\
\EExt^p (\Z, \RPic (X))\ar[r]\ar@{->>}[d]_{\pi}&H^{p-1}_\eh(X,\G_m)\ar[r]^{r_\varpi^{p, 1}}& H^{p,1}_\varpi (X)\\
\Hom (\Z, \RA{p}(X))\ar@{.>}@/_1.6pc/[urr]_-{\theta_\varpi^p}&&}\]
obtained by the composition of \eqref{albmap} with the period regulator $r_\varpi^{p, 1}$,
is therefore sending $\Ext (\Z, \RA{p-1} (X))$ to zero, since  $H^{p,1}_\varpi (X)$ is finitely generated.  We then get the induced mapping $\theta_\varpi^p$ as indicated in the diagram.
\end{proof}
Also for the Betti realization, there is an integrally defined group homomorphism
$$\theta^p_\Z: T_{\Z}(\RA{p}(X)_\C)_\fr \to H^p(X_\an, \Z_\an (1))_\fr$$ 
induced via Cartier duality, by applying the Betti realization $\beta_\sigma$ in \eqref{BettiR} to the motivic Albanese \eqref{albmap} in a canonical way. This is justified after the natural identification of Deligne's $T_\Z$ with the Betti realization $\beta_\sigma$ on 1-motives (see \cite[Thm. 15.4.1]{BVK} and \cite{VO} for an explicit construction of the natural isomorphism $T_\Z\cong \beta_\sigma\Tot $). Rationally, it yields an injection 
$$\theta^p_\Q: T_{\Q}(\RA{p}(X)_\C) \cong H^p_{(1)}(X_\an, \Q_\an (1))\subset H^p(X_\an, \Q_\an (1))$$ where the notation $H^p_{(1)}$ is taken to indicate the largest 1-motivic part of $H^p(X_\an, \Q_\an (1))$ (more precisely, this is given by the underlying $\Q$-vector space associated to the mixed Hodge structure, see \cite[Cor. 15.3.1]{BVK}).

For the de Rham realization, similarly, we have a $K$-linear mapping $$\theta^p_\dr: T_{\dr}(\RA{p}(X))\to H^p_\dr(X).$$ Actually, for ${\sf M} = \LA{p}(X)$ and ${\sf M}^* = \RA{p}(X)$, we have $\eta_{{\sf M}^*}$  the $\C$-inverse of the period isomorphism $\varpi_{{\sf M}^*,\C}$ in Theorem \ref{BettideRham} and  $\eta^{p,1}_X$ which is the inverse of the period isomorphism in Definition \ref{classper}. Together with $\theta^p_\Z$ and $\theta^p_\dr$, we obtain a diagram 
 \[\xymatrix{ T_{\Z}(\RA{p}(X)_\C)_\C \ar[rr]^{\theta^p_\Z\otimes \C} &&H^p(X_\an, \Z_\an (1))_\C \ar[r]^{\ \ 2\pi i} &  H^p(X_\an, \C)\\
T_{\dr}(\RA{p}(X))_{\C}\ar[u]^{\eta_{{\sf M}^*}}\ar[rr]^{\theta^p_\dr\otimes \C} && H^p_\dr(X)\otimes_K \C \ar[u]^{\eta^{p,1}_X}\ar[ur]_{\eta^{p,0}_X}. &
 }\] 
We have that this diagram commutes, in fact:
\begin{lemma}\label{omega} 
Let $X$ be over the field $K = \bar \Q$ and $p\in \Z$. There is a morphism 
$$\theta^p_\drb\df (\theta^p _\dr, \theta^p_\Z): T_{\drb} (\LA{p}(X))_\fr\to H_{\drb}^{p, 1}(X)_\fr$$ in the category $\cMod^{{\cong, \fr}}$.  Rationally 
$\theta^p_\drb\otimes\Q$ becomes injective. Moreover, $\theta^0_\drb$ and $\theta^1_\drb$ are integrally defined isomorphisms.
 \end{lemma}
\begin{proof} This is a consequence of \cite[Cor. 16.3.2]{BVK}. For $p=0, 1$ it is straightforward that they are isomorphisms. \end{proof}

\begin{lemma}\label{regfact}
The map $\theta_\varpi^p$ defined in Lemma \ref{reproj} factors through the de Rham-Betti realization via the Cartier duals \eqref{Cdualtor}, \ie we have the following factorization \[
\xymatrix{
\Hom (\Z, \RA{p}(X))\ar@/^2.5pc/[rrr]^-{\theta_\varpi^p}\ar@{=}[r]&\Hom (\LA{p}(X), \G_m[-1])\ar[r]^{\hspace*{1cm}T_\drb}& T^{\varpi}(\LA{p}(X))\ar[r]^{\iota}& H^{p,1}_\varpi (X)\\
}\]
such that $\iota$ is given by $\theta^p_\drb$ in Lemma \ref{omega}, using Corollary \ref{perhombis}, as follows 
$$T^{\varpi}(\LA{p}(X))=\Hom (\Z (0), T_{\drb}(\LA{p}(X))) \to  \Hom (\Z (0), H_{\drb}^{p, 1}(X))\cong H^{p,1}_\varpi (X)$$ and  the latter $\Hom$ is here taken in $\cMod^{\cong}$. 
\end{lemma}
\begin{proof} By construction $\theta_\varpi^p$ is induced by $r_\varpi^{p, 1}$ on a quotient via the motivic Albanese \eqref{albmap} applying Betti and de Rham realizations so that the claimed factorization is clear.\end{proof}
Thus, showing the period conjecture \eqref{periodconj} for $q=1$ is equivalent to seeing that $\theta_\varpi^p$ is surjective, rationally.
Recall (see \cite[Prop. 10.4.2]{BVK}) that for any $X$ of dimension $d= \dim (X)$ the 1-motive $\LA{d+1}(X)$ is a group of multiplicative type and 
$$\LA{p}(X) =
\begin{cases}
\relax 0 & \text{if $p < 0$}\\
\relax [\Z[\pi_0(X)]\to 0]& \text{if $p = 0$}\\
\relax [L_1\by{u_1} G_1] & \text{if $p= 1$}\\
\relax 0 & \text{if $p > \max (2, d  +1)$}
\end{cases}$$
where $G_1$ is connected, so that $\LA{p}(X)\in \M$ is free for $p=0,1$ (see \cite[Prop. 12.6.3 c)]{BVK}). Thus $\RA{0}(X) =[\Z[\pi_0(X)]\to 0]^*= [0\to \Z[\pi_0(X)]^\vee\otimes \G_m]$ is a torus and we have that $\Ext (\Z, \RA{0}(X)) = \Hom_K (\Z ,\RA{0}(X))= K^*\otimes_\Z \Z[\pi_0(X)]^\vee$  (see \cite[Prop. C.8.3 (b)]{BVK}). 
\begin{thm}\label{per:1:1} For any $X$ over $K =\bar \Q$ we have that \eqref{periodconj} holds true for $p=q=1$, \ie the period regulator $r_\varpi^{1, 1} : H^{0}_\eh(X,\G_m)\onto  H^{1,1}_\varpi (X)$ is surjective. Moreover, considering the 1-motive $\RA{1}(X) = [L_1^*\by{u_1^*} G_1^*]$ which is the Cartier dual of $\LA{1} (X)$ we have a canonical isomorphism
$$\ker u_1^* \cong H^1_\dr(X)\cap H^1(X_\an, \Z_\an (1))=H^{1,1}_\varpi (X).$$
In particular, if $X$ is proper $H^{0}_\eh(X,\G_m)\cong K^*\otimes_\Z \Z[\pi_0(X)]^\vee$ and  $H^1_\dr(X)\cap H^1(X_\an, \Z (1))=0$.
\end{thm}
\begin{proof} In fact, $\RA{1}(X)$ is free and therefore $\Hom (\LA{1}(X), \G_m[-1])\cong \Hom (\Z, \RA{1}(X))\cong  \ker u_1^*$. Thus the extension in Lemma \ref{reproj} is
$$0\to K^*\otimes_\Z \Z[\pi_0(X)]^\vee\to \EExt^1 (\Z, \RPic (X))\to\ker u_1^*\to 0.$$
Moreover $\theta_\varpi^1:\Hom (\Z, \RA{1}(X))\cong \ker u_1^*\by{\simeq} H^{1,1}_\varpi (X)$ is an isomorphism, which in turn implies that $r_\varpi^{1, 1}$ is a surjection. Actually, see Lemma \ref{regfact}, $\theta_\varpi^1$ factors as follows
$$\Hom (\Z, \RA{1}(X))\stackrel{(a)}{\cong} T^{\varpi}(\LA{1}(X))\stackrel{(b)}{\cong}  \Hom (\Z (0), H_{\drb}^{1, 1}(X))\stackrel{(c)}{\cong} H^{1,1}_\varpi (X)$$
where: (a) is the isomorphism obtained applying Corollary \ref{keru} to ${\sf M}=\LA{1} (X)$; (b) is the  $\Hom (\Z (0), - )$ of the isomorphism $\theta^1_\drb : T_{\drb} (\LA{1}(X))\cong H_{\drb}^{1, 1}(X)$ given by $p=1$ in Lemma \ref{omega}; (c) is the isomorphism in Corollary \ref{perhombis}.  If $X$ is proper then $L_1^*=0$, \ie $\LA{1}(X) = [L_1\by{u_1} G_1]$ with $G_1$ an abelian variety (see \cite[Cor. 12.6.6]{BVK}) in such a way that $\RA{1} (X) = [0 \to G_1^*]$, and $H^{0}_\eh(X,\G_m)\cong \G_m (\pi_0(X))$ (see \cite[Lemma 12.4.1]{BVK}).
\end{proof}
\begin{remark} We may actually compute $\RA{1}(X)$ by using descent. For example, if  $X$ is normal let $\bar X$ be a normal compactification of $X$, $p: X_\d \to X$ a smooth hypercovering and $\bar X_{\d}$ a smooth
compactification with normal crossing boundary $Y_{\d}$ such that  $\bar p : \bar X_{\d}\to
\bar X$ is a hypercovering. Then $\bar p^* : \Pic_{\bar X/K}^0\by{\simeq} \Pic_{\bar X_\d/K}^0$ is an abelian variety and  $$\RA{1} (X) = [\Div^0_{Y_\d}(\bar X_\d)\by{u_1^*} \Pic_{\bar X/K}^0]$$ where $\Div^0_{Y_\d}(\bar X_\d) \df \ker (\Div^0_{Y_0}(\bar X_0)\to \Div^0_{Y_1}(\bar X_1))$ (see \cite[Prop. 12.7.2]{BVK}).
\end{remark}
For $X$ smooth we have that (see \cite[Cor. 9.2.3]{BVK})
$$\LA{p}(X) =
\begin{cases}
\relax [\Z[\pi_0(X)]\to 0]& \text{if $p = 0$}\\
\relax [0\to\cA_{X/K}^0] & \text{if $p= 1$}\\
\relax [0\to \NS_{X/K}^*] & \text{if $p= 2$}\\
0 & \text{otherwise,}
\end{cases}$$
where $\cA_{X/K}^0$ is the Serre-Albanese semi-abelian variety and $\NS_{X/K}^*$ denotes the group of multiplicative type dual to the N\'eron-Severi group $\NS_{X/K}$. In this case, we then have 
$$\RA{p}(X) =
\begin{cases}
\relax [0\to \Z[\pi_0(X)]^*]& \text{if $p = 0$}\\
\relax [\Div^0_{Y}(\bar X)\by{u_1^*} \Pic_{\bar X/K}^0] & \text{if $p= 1$}\\
\relax [\NS_{X/K}\to 0] & \text{if $p= 2$}\\
0 & \text{otherwise,}
\end{cases}
$$
for a smooth compactification $\bar X$ with normal crossing boundary $Y$.
Note that, reducing to the smooth case by blow-up induction we can see that the map \eqref{albmap} is an isomorphism for $p=0, 1$ (\cf \cite[Lemma 12.6.4 b)]{BVK}). We deduce the following: 
\begin{cor} For any scheme $X$ over $K =\bar \Q$ we have a short exact sequence
$$0\to K^*\otimes_\Z \Z[\pi_0(X)]^\vee\to H^{0}_\eh(X,\G_m)\longby{r_\varpi^{1, 1}}H^{1,1}_\varpi (X)\to 0.$$
\end{cor}
In general, we also have: 
\begin{propose} \label{reg1}
For $K = \bar \Q$ the period regulator  $r_\varpi^{p,1}$ admits a factorization
$$H^{p, 1}(X)_\Q\cong H^{p-1}_\eh(X,\G_m)_\Q \onto T^{\varpi}_\Q(\LA{p}(X)) \into H^p_\dr(X)\cap H^p(X_\an, \Q_\an (1)) = H^{p,1}_\varpi (X)_\Q$$
where the projection is given by Lemma \ref{reproj} via $T_\drb^\Q$ and the inclusion is given by $\theta^p_\drb \otimes\Q$ in Lemma \ref{omega}. Therefore, the conjecture \eqref{periodconj} is equivalent to $T^{\varpi}_\Q(\LA{p}(X))\cong H^{p,1}_\varpi (X)_\Q$.
\end{propose}
\begin{proof}
In fact,  using the adjunction \eqref{adjoint}, the Cartier dual $\pi^*$ of $\pi$ in Lemma \ref{reproj}, the factorization of Lemma \ref{regfact} and Theorem \ref{thm:ff1mot} we have the following commutative diagram
\[
\xymatrix{
\Hom_{D^b(\M^\Q)} (\LAlb (X), \G_m[-1][p])\ar[r]^{\hspace*{1.2cm}\simeq}\ar@{->>}[d]_{\pi^*\otimes \Q}&H^{p-1}_\eh(X,\G_m)_\Q\ar[r]^{r_{\varpi ,\Q}^{p, 1}}& H^{p,1}_\varpi (X)_\Q\\
\Hom (\LA{p}(X), \G_m[-1])_\Q\ar@{^{(}->}[rru]_-{}\ar[r]_-{T_\drb^\Q}^{\hspace*{1.2cm}\simeq}&T^{\varpi}_\Q(\LA{p}(X))\ar@{^{(}->}[ru]_{}.&}\]
\end{proof}
For $X$ smooth we further have that $$H^{p, 1}(X)\cong H^{p-1}_\eh(X,\G_m)\cong H^{p-1}_\et (X,\G_m)$$ and this latter is vanishing after tensoring with $\Q$ for all $p\neq 1, 2$ (see \cite[Prop. 1.4]{GDix}).
Accordingly, the period conjecture \eqref{periodconj} for $X$ smooth and $p\neq 1, 2$ is in fact equivalent to \eqref{vanishing}, \ie 
\begin{equation}\label{vanishingq=1}H^{p,1}_\varpi (X) = H^p_\dr(X)\cap H^p(X_\an, \Q_\an (1))=0 \quad p\neq 1,2.\end{equation}
For $p=2$ and $X$ smooth we have that $H^{2, 1}(X) \cong \Pic (X)$, $r_\varpi^{2,1}=c\ell$ is induced by the usual cycle class map and $T^{\varpi}_\Q(\LA{2}(X)) = \NS (X)_\Q$. 

We here recover the results of Bost-Charles (see \cite[Thm. 5.1]{Bo} and \cite[Cor. 3.9-3.10]{BC}) as follows. We refer to \cite[Chap. 4]{BVK} for the notion of biextension of 1-motives. The following is a generalization of \cite[Thm. 3.8 2)]{BC} and of the discussion of the sign issue in \cite[\S 3.4]{BC}: 
\begin{lemma} \label{biext}
For ${\sf N}, {\sf M}\in \M (\bar \Q)$ we have that
$$\Biext ({\sf N}, {\sf M};\G_m ) \cong (T_\bdr ({\sf N})^\vee\otimes T_\bdr ({\sf M})^\vee\otimes \Z(1))_\varpi$$
and, when ${\sf N} ={\sf M}$, the subgroup of symmetric biextensions corresponds to alternating elements. 
\end{lemma}
\begin{proof} Recall that  $\Biext ( - , {\sf M};\G_m )$ is representable by the Cartier dual ${\sf M}^*$ for ${\sf M}\in \M (K)$ (see \cite[Prop. 4.1.1]{BVK}).
Thus $\Biext ({\sf N}, {\sf M};\G_m ) = \Hom ({\sf N}, {\sf M}^*)\cong \Hom (T_\bdr ({\sf N}), T_\bdr ({\sf M})^*)$ where we here use Theorem \ref{Cartier}  and  Theorem \ref{thm:ff1mot}. Now $T_\bdr ({\sf M})^*= T_\bdr ({\sf M})^\vee(1)$ in such a way that $\Hom (T_\bdr ({\sf N}), T_\bdr ({\sf M})^*)= \Hom (\Z (0), T_\bdr ({\sf N})^\vee\otimes T_\bdr ({\sf M})^\vee\otimes \Z(1))$ making use of the tensor structure of the category $\Mod^{\cong, \fr}$ by Lemma \ref{tensorper}. 

Assume  ${\sf N} ={\sf M}$. Since $\Biext ({\sf M}, {\sf M};\G_m ) \cong \Hom (T_\bdr ({\sf M}), T_\bdr ({\sf M})^*)$, any biextension $\mathcal P$  corresponds to a pairing $T_\bdr ({\sf M})\otimes T_\bdr ({\sf M})\to  \Z(1)$ which induces the pairing \cite[10.2.3]{De} on Deligne-Hodge realizations and the pairing \cite[10.2.7]{De} on de Rham realizations; if $\mathcal P$ is symmetric, the pairing is alternating by \cite[10.2.5 \& 10.2.8]{De}.
\end{proof}
\begin{cor} For $X$ over $K=\bar \Q$ we have that 
$$\Biext (\LA{1}(X), \LA{1}(X);\G_m )^{\rm sym} \cong (H^{1, 0}_\drb(X)\otimes H^{1, 0}_\drb(X) \otimes \Z(1))_\varpi^{\rm alt}.$$
\end{cor}
\begin{proof} Applying Lemma \ref{biext} to the free 1-motive $\LA{1}(X)$ we obtain the claimed formula. In fact, recall that $T_\Z (\LA{1}(X))\cong H_1(X_\an , \Z_\an)_\fr$ and observe that $H^{1, 0}_\drb(X)$ is identified with $T_\bdr (\LA{1}(X))^\vee$ up to inverting the period isomorphism by the same argument of Lemma \ref{omega}. 
\end{proof}
This implies that the period conjecture for $p=2$ holds true in several cases, \eg for abelian varieties, as previously proved by Bost (see \cite[Thm. 5.1]{Bo}). 

\subsection{The case  $q=0$} 
Consider the case of $\Z (0)$ which is canonically identified with $\Tot ([\Z\to 0])=\Z [0]$.
Note that $H^{p,0}(X)\cong H^{p}_\eh(X,\Z)$. Let $\A (K)\subset \M (K)$ be the full subcategory of $0$-motives or Artin motives over $K$. 
Recall that the motivic $\pi_0$ (see \cite[\S 5.4]{BVK} and \cite[Cor. 2.3.4]{BVA}) is a triangulated functor
$$L\pi_0:\DM^\eff_{\rm gm}\to D^b(\A)$$ whence $L\pi_0(X)\in D^b(\A)$,  a complex in the derived category of Artin motives, associated to the motive of $X$. We have that 
$M(X) \to \Tot L\pi_0(X)\in \DM^\eff_\et$ (see \eqref{tot} for $\Tot$) induces 
$$\Hom_{D^b(\A)} (L\pi_0(X), \Z[p])\to \Hom_{\DM^\eff_{\et}} (M(X), \Z (0) [p])\cong H^{p}_\eh(X,\Z).$$
This map is  an isomorphism, integrally, for $p=0, 1$ (\cf \cite[Lemma 12.6.4 b)]{BVK}) and it becomes, by adjunction, a $\Q$-linear isomorphism, for all $p$.  Recall that for any $M\in \DM^\eff_{\rm gm}$ we have (see \cite[Prop. 8.2.3]{BVK})
$$\LAlb(M(q)) \cong
\begin{cases}
\relax L\pi_0(M)(1)& \text{if $q = 1$}\\
0 & \text{for $q\geq 2$}
\end{cases}$$
where an Artin motive twisted by one is a 1-motive of weight $-2$, \ie the twist by one functor $(-)(1): D^b(\A)\to D^b(\M)$ is induced by $L\leadsto [0\to L\otimes \G_m]$. Note that as soon as $K = \bar \Q$ Artin motives are of homological dimension $0$ and we have that $$\Hom_{D^b(\A)} (L\pi_0(X), \Z[p])= \Hom_{\A} (L_p\pi_0(X), \Z).$$ 
Moreover,  we have that $$H^{p}_\eh(X,\Z)\cong\Hom_{\DM^\eff_{\et}} (M(X)(1), \Z (1) [p])$$
by Voevodsky's cancellation theorem \cite{VC}.
\begin{thm} \label{per:1:0}
For any $X$ over $K =\bar \Q$ we have that \eqref{periodconj} holds true for $p=1$ and $q=0$. Moreover, we have 
$$H^{1}_\et(X,\Z)\cong H^1_\dr(X)\cap H^1(X_\an, \Z_\an )\cong H^{1,0}_\varpi(X)$$ 
which is vanishing if $X$ is normal.
\end{thm}
\begin{proof} Making use of Proposition \ref{zerotwist} we are left to show the period conjecture for $M (X)(1)$ in degree $1$ and twist $1$. We have that
$$\Hom_{\DM^\eff_{\et}} (M(X)(1), \Z (1) [1])\cong \Hom_{D (\M)} (\LAlb (M(X)(1)), \Z (1) [1]).$$
We have $\LA{0}(M (X)(1)) \cong L_0\pi_0(X)(1)\cong [0 \to \Z[\pi_0 (X)]\otimes \G_m]$ in such a way that $$\Ext_{{}_t\M} (\LA{0}(M (X)(1)), \G_m[-1])=0$$ and (\cf \eqref{adjoint} for $M(X)(1)$) we obtain
$$H^{1}_\eh(X,\Z)\cong\Hom_{{}_t\M} (\LA{1}(M (X)(1)), \G_m[-1]).$$
Now 
$T^\varpi (\LA{1}(M (X)(1)))\cong H^{1,1}_\varpi (M(X) (1))\cong H^{1,0}_\varpi (X)$ by Lemma \ref{omega} twisted by $(-1)$ and the same argument in the proof of Theorem \ref{per:1:1} applies here. 
Finally, recall that $H^{1}_\et(X,\Z)\cong H^{1}_\eh(X,\Z)$ for any scheme $X$ and $H^{1}_\eh(X,\Z)=0$ if $X$ is normal (see \cite[Lemma 12.3.2 \& Prop. 12.3.4]{BVK}). 
\end{proof}
\begin{remark} For $X$ not normal (\eg for the nodal curve) the group $H^{1}_\et(X,\Z)$ can be non-zero. Moreover, for any $X$ we have a geometric interpretation $H^{1}_\et(X,\Z) \cong L\Pic (X)\into \Pic (X[t, t^{-1}])$ by a theorem of Weibel \cite[Thm. 7.6]{Weib}. Note that this $L\Pic (X)$ is also a sub-quotient of the negative $K$-theory group $K_{-1}(X)$ (see \cite[Thm. 8.5] {Weib}). 
\end{remark}
For $X$ smooth we have a quasi-isomorphism $L\pi_0(X)\cong \Z[\pi_0(X)][0]$ (see \cite[Prop. 5.4.1]{BVK}) which means that $H^{p,0}(X)_\Q=0$ for $p\neq 0$. This yields (as it also does  Proposition \ref{perconjsm} for $X$ smooth) that the period conjecture \eqref{periodconj} is equivalent to 
\begin{equation}\label{vanishingq=0}H^p_\dr(X)\cap H^p(X_\an, \Q_\an)=0 \quad p\neq 0.\end{equation} 
\begin{remark} The period conjecture \eqref{periodconj} for $q=0$ and $X$ smooth is also equivalent to the surjectivity of $f^p_\varpi: H^{p,0}_\varpi(\pi_0(X))_\Q\to H^{p,0}_\varpi(X)_\Q$ induced by the canonical morphism $f: X\to \pi_0(X)$,  for all $p\geq 0$. In fact, the morphism $f$ induces a map $M(X)\to M(\pi_0 (X))$ and a commutative square
by functoriality  
 \[\xymatrix{ H^{p,0}(X)_\Q \ar[r]^{r_\varpi^{p,0}} \ar@{=}[d]_{f^p} & H^{p,0}_\varpi(X)_\Q\\
H^{p,0}(\pi_0(X))_\Q\ar[r]^{\cong}& H^{p,0}_\varpi(\pi_0(X))_\Q\ar[u]_{f^p_\varpi}
 }\]
where $f^p: \Hom_{\DM^\eff_{\et}} (M(\pi_0(X)), \Z[p])_\Q\to \Hom_{\DM^\eff_{\et}} (M(X), \Z  [p])_\Q$ is an isomorphism for $X$ smooth; since $\dim \pi_0 (X) =0$ then $r^{p,0}_\varpi$ is clearly an isomorphism for $\pi_0 (X)$. For $p=0$ the group $H^{0}(X_\an,\Z_\an (0))$ has rank equal to the rank of $\Z[\pi_0(X)]$ and $f^0_\varpi$ is an isomorphism; for $p\neq 0$ the surjectivity of $f^p_\varpi$ is equivalent to the vanishing of all groups. 
\end{remark}

\subsection{Arbitrary twists}
We now apply Waldschmidt's Theorem \ref{thm:Wal} to arbitrary twists.
\begin{propose}\label{prop:odd} For ${\sf M}=[{\sf L}\to {\sf G}]$ a free $1$-motive over $K=\overline{\Q}$ and $q\in \Z$ an integer we have that  
\begin{itemize}
\item[1)] the group $\Hom \bigl(\Z(q), T_{\bdr}({\sf M})\bigr)$ of homomorphisms in $\Modc$ or $\QMod^{\cong}$  is trivial for $q\neq 0$, $1$;\smallskip
\item[2)] the group $\Hom \bigl(\Z(q), T_{\drb}({\sf M})\bigr)$ of homomorphisms in $\cMod^{\cong}$ or $\cQMod^{\cong}$  is trivial for $q\neq 0$, $1$.\smallskip
\end{itemize}
\end{propose}
\begin{proof} 1) We work in $\Modc$ and leave  the other case to the reader. We suppose first that ${\sf L}=0$. Consider a non trivial $\phi\in \Hom_{\bdr}\bigl(\Z(q), T_{\bdr}({\sf M})\bigr)$ and the
subgroup $\Gamma=T_\Z(\Z(q))=\Z \subset T_{\dr}(\Z(q))_\C=\C$. Via the non trivial map $\phi_{K} \otimes \C\colon T_{\dr}(\Z(q))_\C \to \T_{\dr}({\sf G})_\C = \Lie ({\sf G}_\C^\natural)$ we can identify $\Gamma$ with a subgroup of $\Lie ({\sf G}_\C^\natural)$. This subgroup is contained in $V_\C$ with $V\subset \Lie ({\sf G}^\natural)$ defined by the
image $\phi_{K}\bigl(T_{\dr}(\Z(q))\bigr)$.  Via the exponential map $ \Lie ({\sf G}_\C^\natural) \to {\sf G}^\natural(\C)$ the image of $\Gamma$ is $0\in {\sf G}^\natural(K)$ as
$\phi$ is a map in the category $\Modc$ (respectively $\QMod^{\cong}$). We deduce from Waldschmidt's Theorem \ref{thm:Wal} that $V\subset \Lie ({\sf G}^\natural)$ is the
Lie algebra of a $1$-dimensional algebraic subgroup $H$ of ${\sf G}^\natural$. There are only two possibilities $H=\G_a$ and $H=\G_m$. In both cases the period morphism
for $\Z(q)$ identifies $\Gamma$ with the subgroup $(2 \pi i )^q \Z \subset \Lie (H_\C)$ that goes to $0$ via $\exp_{H_\C}$. For $H=\G_a$ the map $\exp_{H_\C}$ is the
identity, leading to a contradiction. For $H=\G_m$ the kernel of $\exp_{H_\C}$ is $(2 \pi i ) \Z$ forcing $q=1$.

Secondly we suppose that ${\sf G}=0$. Consider a non trivial $\phi\in \Hom_{\bdr}\bigl(\Z(q), T_{\bdr}({\sf M})\bigr)$. Recall that $\T_{\dr}({\sf M})={\sf L} \otimes K$ and the period map
is induced by the inclusion ${\sf L} \subset {\sf L} \otimes K$. Let $e=\phi_K(1)\in {\sf L} \otimes K$. It is a non-zero element. Using that $T_\Z(\Z(q))$ is identified via the period
morphism for $\Z(q)$ with  $(2 \pi i )^q \Z$, we deduce that $\phi_\Z(1) = (2 \pi i )^q \cdot e$ should lie in ${\sf L}\subset {\sf L}\otimes K$. As $\pi$ is transcendental,
this forces $q=0$.

For general ${\sf M} = [{\sf L}\to {\sf G}]$ we reduce to ${\sf G}$ and ${\sf L}$ to conclude the statement.

2) We prove the statement for $\cMod^{\cong}$ using Lemma \ref{varsigma}. The analogue for $\cQMod^{\cong}$  follows similarly. Given a $1$-motive ${\sf M}$ and its Cartier dual ${\sf M}^*$ we have a
natural identification $$\varsigma\colon \Hom \bigl(\Z(q), T_{\bdr}({\sf M}^*)\bigr)\cong \Hom \bigl(\varsigma (\Z(q)), \varsigma (T_{\bdr}({\sf M}^*))= \Hom \bigl(\Z(q), T_{\drb}({\sf M})\bigr) \bigr).$$

The statement follows then from 1).
\end{proof}
 Denote $H^{p,q}_{\drb , (1)}(X)_\fr \subset H^{p, q}_{\drb}(X)_\fr$ the image of  $T_{\drb}(\LA{p}(X))_\fr (q-1)$ under $\theta_\drb^p (q-1)$ of Lemma \ref{omega} twisted by $q-1$. We have:
\begin{cor} \label{cor:per:1}
We get that $H^{p,q}_{\varpi ,(1)} (X)_\fr = 0$ if $q\neq 0, 1$. For $p=1$ we have $H^{1,q}_{\varpi ,(1)} (X) = H^{1,q}_{\varpi} (X)$
and 
$$H^{1,q}_{\varpi} (X) =
\begin{cases}
\relax H^1_\et(X, \Z)\ \ \text{(see Theorem \ref{per:1:0})} & \text{if $q = 0$}\\
\relax \ker  u_1^*\ \ \text{(see Theorem \ref{per:1:1})} & \text{if $q= 1$}\\
0 & \text{$q\neq 0, 1$.}
\end{cases}
$$
\end{cor}
\begin{proof}
We apply Proposition \ref{prop:odd} 2) to ${\sf M} = \LA{p}(X)$ to deduce that  $$H^{p,1-q}_{\varpi ,(1)} (X)_\fr = \Hom_{\drb}\bigl(\Z (0), H^{p,1}_{\drb, (1)}(X)(-q))\bigr) = \Hom_{\drb}\bigl(\Z(q), T_{\drb}(\LA{p}(X))_\fr\bigr)=0$$if $q\neq 0, 1$.
\end{proof}
Thus, for the period conjecture in degree $p=1$, the previous computations for the twists $q=0, 1$ are the only relevant. 

\subsection{Higher odd degrees}
Next, let $X$ be a smooth and projective variety over $K = \bar \Q$. Denote $J^{2k+1}(X)$ the  intermediate Jacobian: as a real analytic manifold, it is defined as the quotient of the image $H^{2k+1}_\Z(X)$ of $H^{2k+1}\bigl(X_\an, \Z_\an(k)\bigr)$ in   $H^{2k+1}\bigl(X_\an,\R(k)\bigr)$. 
This defines a full lattice of $H^{2k+1}\bigl(X_\an,\R(k)\bigr)$ so that  $J^{2k+1}(X)$ is compact. It has also a natural complex analytic structure induced by the
identification $$ J^{2k+1}(X)\df H^{2k+1}(X_\an,\C)/\bigl(F^{k+1} H^{2k+1}(X_\an,\C) + (\varpi_X^{2k+1,k})^{-1}\bigl(H^{2k+1}_\Z(X)\bigr) \bigr).$$Thus $J^{2k+1}(X)$
is a complex torus.

For integers $n$ define $N^n H^{2k+1}\bigl(X_{\rm an}, \Q_\an(k)\bigr)\subset H^{2k+1}\bigl(X_{\rm an}, \Q_\an(k)\bigr)$, the $n$-th step of the geometric coniveau
filtration, as the kernel of $$H^{2k+1}\bigl(X_{\rm an}, \Q_\an(k)\bigr)\longrightarrow \bigoplus_{Z\subset X} H^{2k+1}\bigl(X_{\rm an}\backslash Z_{\rm an},
\Q_\an(k)\bigr)$$ for $Z\subset X$ varying among the codimension $\geq n$ closed subschemes.
\begin{lemma}\label{prop:algebraicIJac} Assume that $H^{2k+1}\bigl(X_{\rm an}, \Q_\an (k)\bigr)$ has geometric coniveau $k$, \ie that we have
$N^k H^{2k+1}\bigl(X_{\rm an}, \C\bigr)=H^{2k+1}\bigl(X_{\rm an}, \C\bigr)$. Then  $J^{2k+1}(X)$ is an abelian variety, which descends to an abelian variety
$J^{2k+1}(X)_K$  over $K$ with $$T_{\drb}\bigl(J^{2k+1}(X)_K\bigr)=\bigl(H^{2k+1}_\dr(X),H^{2k+1}_\Z(X), \eta_X^{2k+1,k}\bigr).$$
\end{lemma}
\begin{proof} Under the assumption, $H^{2k+1}_\Z(X)$ is a polarized Hodge structure of type $(1,0)$ and $(0,1)$ so that $J^{2k+1}(X)$ is polarizable and, hence,
an abelian variety.  The second statement follows from  \cite[Thm. A]{ACMV2} where it is proven that there exists an abelian variety $J$ over $K$ and a
correspondence $\Gamma \in {\rm CH}^{h}(J\times_K X)$ over $K$, for $h=k+\dim J^{2k+1}(X)$, inducing an isomorphism $\Gamma_\ast\colon H^1(J_\an,\Q_\an) \cong
H^{2k+1}\bigl(X_\an, \Q_\an(k)\bigr)$ (and hence in de Rham cohomology, compatibly with the period morphisms). Then set $J^{2k+1}(X)_K \df J$.
\end{proof}
The period conjecture \eqref{periodconj} in odd degrees for $X$ predicts that $H^{2k+1,q}_\varpi(X)= H^{2k+1}_\dr(X)\cap H^{2k+1}\bigl(X_{\rm an}, \Z_\an (q)\bigr)_\fr =0$ for every $k\in \N$ and every $q\in\Z$.
\begin{propose}\label{prop:oddPC} The period conjecture \eqref{periodconj} in  degree $p=2k+1$ and any twist $q$ for $X$ smooth and projective holds true if $H^{2k+1}\bigl(X_{\rm an}, \Q_\an (k)\bigr)$ has geometric coniveau $k$.\end{propose}
\begin{proof} Thanks to Lemma \ref{prop:algebraicIJac} we have that 
$$H^{2k+1,q}_\varpi(X)=
\Hom \bigl(\Z(0), T_{\drb}\bigl(J^{2k+1}(X)_K\bigr)(q-k)\bigr)=\Hom \bigl(\Z(k-q), T_{\drb}\bigl(J^{2k+1}(X)_K\bigr)\bigr).$$ This is trivial for $k-q\neq 0$,
$1$ by Proposition \ref{prop:odd}. Now use Theorem \ref{thm:ff1mot}. For $k-q=0$ we get that this coincides with the homomorphisms of $1$-motives  from $[\Z \to 0]$ to $[0\to J^{2k+1}(X)_K]$, which is $0$.
For $k-q=1$ this coincides with the homomorphisms of $1$-motives from $[0\to \G_m]$ to $[0\to J^{2k+1}(X)_K]$, which is also $0$.
\end{proof}

\begin{remark} Lemma \ref{prop:algebraicIJac} is proven more generally in  \cite{ACMV2} for the Hodge structure $N^k H^{2k+1}\bigl(X_{\rm an},
\Q_\an(k)\bigr)\subset H^{2k+1}\bigl(X_{\rm an}, \Q_\an(k)\bigr)$ defined by the $k$-th step of the coniveau filtration. Namely, if $X$ is defined over a number
filed $L\subset K$, there is an abelian variety $J_a^{2k+1}(X)$ over $L$ with $T_{\drb}\bigl(J^{2k+1}_a(X)\bigr)=\bigl(N^kH^{2k+1}_\dr(X), N^k H^{2k+1}_\Z(X),\eta_X^{2k+1,k}\bigr)$. The proof of Proposition \ref{prop:oddPC} using $J_a^{2k+1}(X)$ gives the following weak version of the period conjecture:
\begin{equation}
N^k H^{2k+1}_\dr(X)\cap N^k H^{2k+1}\bigl(X_{\rm an}, \Z_\an (q)\bigr)_\fr=0\ \  \text{for every $k\in \N$ and every $q\in\Z$.}
\end{equation}

The assumption in Lemma \ref{prop:algebraicIJac} amounts to saying that $N^k H^{2k+1}\bigl(X_{\rm an}, \Q_\an(k)\bigr)$ is equal to $H^{2k+1}\bigl(X_{\rm an},
\Q_\an(k)\bigr)$. 
This equality holds, for example, for $k=1$ for uniruled smooth projective threefolds; see \cite{ACMV1}. 

The assumption implies, and under the generalized Hodge conjecture is equivalent to, the fact that the Hodge structure $H^{2k+1}\bigl(X_{\rm an}, \Q\bigr)$ has Hodge coniveau $k$, i.e.,
$H^{2k+1}\bigl(X_{\rm an}, \C\bigr)$ is the sum of the $(k+1,k)$ and $(k, k+1)$ pieces of the Hodge decomposition.   Under this weaker condition on the Hodge
coniveau one can still prove that $H^{2k+1}\bigl(X_{\rm an}, \Q_\an(k)\bigr)$ is the Hodge structure associated to the abelian variety $J^{2k+1}_{\rm alg}(X)$ over
$\C$, called the algebraic intermediate Jacobian in $J^{2k+1}(X)$. Unfortunately one lacks the descent to $K$. See the discussion in \cite{ACMV1}.
\end{remark} 

\appendix

\section{Divisibility properties of motivic cohomology {\rm (by B. Kahn)}}\label{KD}

In this appendix, some results of Colliot-Th\'el\`ene and Raskind on the $\sK_2$-cohomology of smooth projective varieties over a separably closed field $k$ are extended to the \'etale motivic cohomology of smooth, not necessarily projective, varieties over $k$. Some consequences are drawn, such as the degeneration of the Bloch-Lichtenbaum spectral sequence for any field containing $k$. 

Recall that in \cite{ctr}, Colliot-Th\'el\`ene and Raskind study the structure of the $\sK_2$-cohomology groups of a smooth projective variety $X$ over a separably closed field. Following arguments of Bloch \cite{bloch}, their proofs use the Weil conjecture proven by Deligne \cite{weilI} and the Merkurjev-Suslin theorem \cite{ms}. These results and proofs can be reformulated in terms of motivic cohomology, since
\[H^i_\Zar(X,\sK_2)\simeq H^{i+2}(X,\Z(2))\]
or even in terms of \'etale motivic cohomology, since
\[H^j(X,\Z(2))\iso H^j_\et(X,\Z(2))\text{ for }j\le 3\]
as follows again from the Merkurjev-Suslin theorem. 

If we work in terms of \'etale motivic cohomology, the recourse to the latter theorem is irrelevant and only the results of \cite{weilI} are needed; in this form, the results of \cite{ctr} and their proofs readily extend to \'etale motivic cohomology of higher weights, as in \cite[Prop. 4.17]{cycletale} and \cite[Prop. 1]{indec} (see also \cite[Prop. 1.3]{RS}).

Here we generalise these results to the \'etale motivic cohomology of smooth varieties over a separably closed field: see Theorem \ref{t2}. This could be reduced by a d\'evissage to the smooth projective case, using de Jong's alteration theorem in the style of \cite{finiteness}, but it is simpler to reason directly by using cohomology with compact supports, and Weil II \cite{weilII} rather than Weil I \cite{weilI}. (I thank H\'el\`ene Esnault and Eckart Viehweg for suggesting to use this approach). This descends somewhat to the case where the base field $k$ is not separably closed, yielding information on the Hochschild-Serre filtration on \'etale motivic cohomology (Theorem \ref{t2.1}). The rest of the appendix is concerned with implications on motivic cohomology of a field $K$ containing a separably closed field: the main result, which uses the norm residue isomorphism theorem of Voevodsky, Rost et al. (\cite{voeann}, see also \cite{hwannals}) is that $H^i(K,\Z(n))$ is divisible for $i\ne n$ (Theorem \ref{t3}). As an immediate consequence, the ``Bloch-Lichtenbaum'' spectral sequence of $K$ from motivic cohomology to algebraic $K$-theory degenerates (Theorem \ref{t4}).  We also show that the cokernel of the cup-product map 
\[H^{i-1}(K,\Z(n-1))\otimes K^*\to H^i(K,\Z(n))\]
is uniquely divisible for $i<n$ (Theorem \ref{t5}).

Throughout, motivic cohomology is understood in the sense of Suslin and Voevodsky (hypercohomology of the Suslin-Voevodsly complexes \cite{SV}).

\subsection{A weight and coniveau argument}

Let $X$ be a separated scheme of finite type over a finitely generated field $k$. 
\begin{propose}\label{p2.1} Let $n\in\Z$, $k_s$ a separable closure of $k$ and $G=Gal(k_s/k)$.
Let $\bar X=X\otimes_k k_s$. Then $H^j_c(\bar X,\Z_l(m))^G$ and $H^j_c(\bar X,\Z_l(m))_G$ are finite for $j\notin [2m,m+d]$ and any  prime number $l$ invertible in $k$, where $d=\dim X$.
\end{propose}

\begin{proof} Suppose first that $k=\F_q$ is finite. By \cite[Cor. 5.5.3 p. 394]{sga7}, the
eigenvalues of Frobenius acting on $H^j_c(\bar X,\Q_l)$ are algebraic integers which are
divisible by $q^{j-d}$ if $j\ge d$. This yields the necessary bound $m\ge j-d$ for an
eigenvalue $1$. On the other hand, by \cite{weilII}, these eigenvalues have archimedean absolute
values $\le q^{j/2}$: this gives the necessary bound $m\le j/2$ for an eigenvalue $1$. 
The conclusion follows.

In general, we may choose a regular model $S$ of $k$, of finite type over $\Spec \Z$, such that
$X$ extends to a compactifiable separated morphism of finite type $f:\sX\to S$. By
\cite[lemma 2.2.2 p. 274 and 2.2.3 p. 277]{sga5}, $R^jf_!\Z_l$ is a constructible $\Z_l$-sheaf
on $S$ and its formation commutes with any base change. Shrinking $S$, we may assume that it is
locally constant and that $l$ is invertible on $S$. For a closed point $s\in S$, this gives an
isomorphism
\[H^j_c(\bar X,\Z_l)\simeq H^j_c(\bar X_s,\Z_l)\]
compatible with Galois action, and the result follows from the first case.
\end{proof}

\begin{cor}\label{c2.1} If $X$ is smooth in Proposition \ref{p2.1}, then $H^i(\bar
X,\Z_l(n))^{(G)}$  is finite for $i\notin [n,2n]$, where the superscript $^{(G)}$ denotes the subset of elements invariant under some open subgroup of $G$. If $X$ is smooth projective, then $H^i(\bar X,\Z_l(n))^{(G)}$ is finite for $i\ne 2n$ and $0$ for almost all $l$.
\end{cor}

\begin{proof} By Poincar\'e duality and Proposition \ref{p2.1}, $H^i(\bar
X,\Z_l(n))^{G}$  is finite for $i\notin [n,2n]$; the claim follows since $H^i(\bar X,\Z_l(n)))$ is a finitely generated $\Z_l$-module. In the projective case, the Weil conjecture \cite{weilI} actually gives the finiteness of $H^i(\bar X,\Z_l(n))^{G}$, hence of $H^i(\bar X,\Z_l(n))^{(G)}$, for all $i\ne 2n$. But Gabber's theorem \cite{gabber} says that $H^i(\bar X,\Z_l(n))$ is torsion-free for almost all $l$, hence the conclusion.
\end{proof}

\begin{thm}\label{t2} Let $X$ be a smooth variety over a separably closed field $k$ of exponential characteristic $p$. Then, for $i\notin [n,2n]$, the group $H^i_\et(X,\Z(n))[1/p]$
is an extension of a direct sum $T$ of finite $l$-groups by a divisible group. If $X$ is projective, this is true for all $i\ne 2n$, and $T$ is finite. If $p>1$, $H^i_\et(X,\Z(n))$ is uniquely $p$-divisible for $i<n$. In particular, $H^i_\et(X,\Z(n))\otimes \Q/\Z=0$ for $i<n$.  For $i\le 1$, $H^i_\et(X,\Z(n))$ is
divisible. The sequence
\[0\to H^{i-1}_\et(X,\Q/\Z(n))\to H^{i}_\et(X,\Z(n))\to
H^{i}_\et(X,\Z(n))\otimes \Q\to 0\]
is exact for $i<n$.
\end{thm}

\begin{proof} 
Away from $p$, it is identical to \cite[proof of Prop. 1]{indec} (which is the projective case) in view of Corollary \ref{c2.1}. The unique $p$-divisibility of $H^i_\et(X,\Z(n))$ for $i<n$ follows from \cite[Th. 8.4]{gl2} and requires no hypothesis on $k$.
\end{proof}

\begin{cor}\label{c2.2} Let $K$ be a field containing a separably closed field $k$. Then,  for $i<n$, the sequence
\[0\to H^{i-1}_\et(K,\Q/\Z(n))\to H^{i}_\et(K,\Z(n))\to
H^{i}_\et(K,\Z(n))\otimes \Q\to 0\]
is exact and the left group has no $p$-torsion if $p=\car K$.
\end{cor}

\begin{proof} We may assume $K/k$ finitely generated. By Theorem \ref{t2}, this is true for any smooth model of $K$ over $k$, and we pass to the limit (see \cite[Prop. 2.1 b)]{bbki}, or rather its proof, for the commutation of \'etale motivic cohomology with limits).
\end{proof}

\begin{remarks} 1) At least away from $p$, the range of ``bad'' $i$'s in Corollary \ref{c2.1} and Theorem \ref{t2} is $[n,2n]$ in general but shrinks to $2n$ when $X$ is projective. If we remove a smooth closed subset, this range becomes
$[2n-1,2n]$. As the proof of Proposition \ref{p2.1} shows, it depends on the length of the weight filtration on $H^*(\bar X,\Q_l)$. If $X=Y-D$, where $Y$ is smooth projective and $D$ is a simple normal crossing divisor with $r$ irreducible components,  the range is $[2n-r,2n]$. It would be interesting to understand the optimal range in general, purely in terms of the geometry of $X$.\\ 

2) Using Proposition \ref{p2.1} or more precisely its proof, one may recover the $l$-local version of \cite[Th.
3]{finiteness} without a recourse to de Jong's alteration theorem. I don't see how to get the global
finiteness of loc. cit. with the present method, because one does not know whether the torsion of $H^j_c(\bar X,\Z_l)$ vanishes for $l$ large when $X$ is not smooth projective.\\ 

3) Using a cycle class map to Borel-Moore $l$-adic homology, one could use
Proposition \ref{p2.1} to extend Theorem \ref{t2} to higher Chow groups of arbitrary
separated $k$-schemes of finite type. Such a cycle class map was constructed in
\cite[\S 1.3]{glrev1}. Note that Borel-Moore $l$-adic cohomology is dual to $l$-adic cohomology with compact supports, so the bounds for finiteness are obtained from those of Proposition \ref{p2.1} by changing signs.
\end{remarks}

\subsection{Descent}

\begin{thm}\label{t2.1} Let $X$ be a smooth variety over a field $k$; write $k_s$ for a separable closure of $k$, $X_s$ for $X\otimes_k k_s$ and $G$ for $Gal(k_s/k)$. For a complex of sheaves $C$ over $X_\et$, write $F^rH^i_\et(X,C)$ for the filtration on $H^i_\et(X,C)$ induced by the Hochschild-Serre spectral sequence
\[E_2^{r,s}(C)=H^r(G,H^s_\et(X_s,C))\Rightarrow H^{r+s}_\et(X,C).\]
Then, for $i<n$, the homomorphism 
\[F^rH^{i-1}_\et(X,\Q/\Z(n))\to F^rH^i_\et(X,\Z(n))\] 
induced by the Bockstein homomorphism $\beta$ is bijective for $r\ge 3$ and surjective for $r=1,2$.
\end{thm}

\begin{proof} By the functoriality of $E_m^{r,s}(C)$ with respect to morphisms of complexes, we have a morphism of spectral sequences
\[\delta_m^{r,s}:E_m^{r,s-1}(\Q/\Z(n))\to E_m^{r,s}(\Z(n))\]
converging to the Bockstein homomorphisms. By Theorem \ref{t2}, $\delta_2^{r,i-r}$ is bijective for $r\ge 2$ and surjective for $r=1$. It follows that, for $m\ge 3$, $\delta_m^{r,i-r}$ is bijective for $r\ge 3$ and surjective for $r=1,2$. The conclusion follows.
\end{proof}

\begin{remarks} 1) Of course, $F^rH^i_\et(X,\Z(n))$ is torsion for $r>0$ by a transfer argument, hence is contained in $\beta  H^{i-1}(X,\Q/\Z(n))$. The information of Theorem \ref{t2.1} is that it equals $\beta  F^rH^{i-1}(X,\Q/\Z(n))$.\\
2) For $i\ge n$, we have a similar conclusion for higher values of $r$, with the same proof: this is left to the reader.
\end{remarks}

\subsection{Getting the norm residue isomorphism theorem into play}

Recall that for any field $K$ and any $i\le n$, we have an isomorphism

\begin{equation}\label{eq2.1}
H^{i}(K,\Z(n))\iso H^{i}_\et(K,\Z(n)).
\end{equation}

Indeed, this is seen after localising at $l$ for all prime numbers $l$. For $l\ne \car K$, this follows from \cite{SV,gl} and the norm residue isomorphism theorem \cite{voeann}, while for $l=\car K$ it follows from  \cite{gl2}. Finally, $H^{i}(K,\Z(n))=0$ for $i>n$. This yields:

\begin{thm}\label{t3} Let $K$ be as in Corollary \ref{c2.2}. Then, for $i\ne n$, the group in \eqref{eq2.1} is divisible. 
\end{thm}

\begin{proof} Again it suffices to prove this statement after tensoring with $\Z_{(l)}$ for all prime numbers $l$. This is an immediate consequence of Corollary \ref{c2.2} since, by  \cite{voeann}, one has an isomorphism for $l\ne \car k$
\[H^{i-1}_\et(K,\Q_l/\Z_l(n))\simeq K_{i-1}^M(K)\otimes \Q_l/\Z_l(n-i+1)\] 
and the right hand side is divisible.
\end{proof}

\subsection{Application: degeneration of the Bloch-Lichtenbaum spectral sequence}

\begin{thm}\label{t4} Let $K$ be as in Corollary \ref{c2.2}. Then the Bloch-Lichten\-baum spectral sequence  \cite[(1.8)]{levine1}
\[E_2^{p,q}=H^{p-q}(K,\Z(-q))\Rightarrow K_{-p-q}(K)\]
degenerates. For any $n>0$, the map $K_n^M(K)\to K_n(K)$ is injective with divisible cokernel.
\end{thm}

\begin{proof} By the Adams operations, the differentials are torsion \cite[Th. 11.7]{levine1}. By Theorem \ref{t3}, they
vanish on the divisible groups $E_2^{p,q}$
for
$p< 0$. But $H^i(K,\Z(n))=0$ for $i>n$, so $E_2^{p,q}=0$ for $p>0$.
The last statement follows from the degeneration plus Theorem \ref{t3}.
\end{proof}

\begin{remarks} 1) Again by the Adams operations, the filtration on $K_n(K)$ induced by the
Bloch-Lichtenbaum spectral sequence splits after inverting $(n-1)!$ for any field $K$. On the other hand, we
constructed in \cite{klocal} a canonical splitting of the
corresponding spectral sequence with finite coefficients, including the abutment; the hypothesis that $K$ contains a separably closed field is not required there. This implies
in particular that the map
\[K_n^M(K)/l^\nu\to K_n(K)/l^\nu\]
is split injective \cite[Th. 1 (c)]{klocal}, hence bijective if $K$ contains a separably closed subfield by Theorem \ref{t4}.  
Could it be that the mod $l^\nu$ splittings of \cite{klocal} also exist
integrally?\\

2) As in \cite{klocal}, Theorems \ref{t3} and \ref{t4} extend to regular semi-local rings of
geometric origin containing a separably closed field; the point is that, for such rings
$R$, the groups $H^{i-1}_\et(R,\Q_l/\Z_l(n))$ are divisible by the universal exactness of the
Gersten complexes (\cite{grayson}, \cite[Th. 6.2.1]{cthk}).
\end{remarks}

\subsection{The map $H^{i-1}(K,\Z(n-1))\otimes K^*\to H^i(K,\Z(n))$}

\begin{thm}\label{t5} Let $K$ be as in Corollary \ref{c2.2}. Then, for $i<n$,
\begin{thlist}
\item The cokernel of the cup-product map
\[H^1(K,\Z(n-i+1))\otimes H^{i-1}(K,\Z(i-1))\by{\gamma^{i,n}} H^i(K,\Z(n))\]
is uniquely divisible. 
\item The cokernel of the cup-product map
\[H^{i-1}(K,\Z(n-1))\otimes K^*\by{\delta^{i,n}} H^i(K,\Z(n))\]
is uniquely divisible. 
\end{thlist}
\end{thm}

\begin{proof} By \eqref{eq2.1}, we may use the \'etale version of these groups. 

(i) Since $H^i_\et(K,\Z(n))$ is divisible
by Theorem \ref{t3}, so is $\Coker \gamma^{i,n}$. Let $\nu\ge 1$ and $l$ prime $\ne \car K$. The diagram
\[\begin{CD}
H^1_\et(K,\Z(n-i+1))\otimes H^{i-1}_\et(K,\Z(i-1))@>\gamma^{i,n}>> H^i_\et(K,\Z(n))\\
@A{\beta\otimes 1}AA @A{\beta}AA\\
H^0_\et(K,\Z/l^\nu(n-i+1))\otimes H^{i-1}_\et(K,\Z(i-1))@>\cup>> H^{i-1}_\et(K,\Z/l^\nu(n))
\end{CD}\]
commutes, where $\beta$ denotes Bockstein. The bottom horizontal map is surjective (even bijective) by the norm residue isomorphism theorem (resp. by \cite{gl2}). By Theorem \ref{t3} again, $H^1_\et(K,\Z(n-i+1))$ is $l$-divisible, hence so
is $H^1_\et(K,\Z(n-i+1))\otimes H^{i-1}_\et(K,\Z(i-1))$, and $\Coker \gamma^{i,n}$ is also
$l$-torsion free by an easy diagram chase.

(ii) Consider the commutative diagram
\[\begin{CD}
H^1_\et(K,\Z(n-i+1))\otimes H^{i-2}_\et(K,\Z(i-2))\otimes K^*@>{\gamma^{i-1,n-1}\otimes 1}>> H^{i-1}_\et(K,\Z(n-1))\otimes K^*\\
@V{1\otimes \cup}VV @V{\delta^{i,n}}VV\\
H^1_\et(K,\Z(n-i+1))\otimes H^{i-1}_\et(K,\Z(i-1))@>\gamma^{i,n}>> H^i_\et(K,\Z(n)).
\end{CD}\]

Since the left vertical map is surjective, we see that $\Coker \delta^{i,n}$ is the quotient of $\Coker \gamma^{i,n}$ by the image of the divisible group $H^{i-1}_\et(K,\Z(n-1))\otimes K^*$ (Theorem \ref{t3}), hence the claim follows from (i).
\end{proof}

It would be very interesting to describe $\Ker \delta^{i,n}$, but this seems out
of range.

\bibliography{bibP1MAFin}
\bibliographystyle{plain}
\end{document}